\documentclass[11pt, twoside]{amsart}

\usepackage{amsmath,amssymb,amsthm,enumitem,stmaryrd,mathrsfs}

\usepackage{thm-restate,hyperref}

\usepackage[pdftex]{graphicx}

\usepackage[top=1.5in, bottom=1.5in, left=1.5in, right=1.5in]{geometry}

\usepackage{tikz}

\allowdisplaybreaks

\theoremstyle{plain}
\newtheorem{theorem}{Theorem}[section]
\newtheorem{lemma}[theorem]{Lemma}
\newtheorem{cor}[theorem]{Corollary}
\newtheorem{prop}[theorem]{Proposition}

\newtheorem*{lemma*}{Lemma}
\newtheorem*{cor*}{Corollary}
\newtheorem*{theorem*}{Theorem}

\theoremstyle{definition}
\newtheorem{example}{Example}
\newtheorem{definition}[theorem]{Definition}

\theoremstyle{remark}

\newtheorem*{fact*}{Fact}
\newtheorem*{remark}{Remark}

\newtheorem*{notation}{Notation}

\let\oldproofname=\proofname
\renewcommand{\proofname}{\rm\bf{\oldproofname}}

\newcommand{\refeq}[1]{(\ref{#1})}

\newcommand{\R}{\mathbb R}
\newcommand{\Q}{\mathbb Q}
\newcommand{\Z}{\mathbb Z}
\newcommand{\C}{\mathbb C}

\newcommand{\phii}{\varphi}

\makeatletter
\providecommand*{\twoheadrightarrowfill@}{%
  \arrowfill@\relbar\relbar\twoheadrightarrow
}
\providecommand*{\twoheadleftarrowfill@}{%
  \arrowfill@\twoheadleftarrow\relbar\relbar
}
\providecommand*{\xtwoheadrightarrow}[2][]{%
  \ext@arrow 0579\twoheadrightarrowfill@{#1}{#2}%
}
\providecommand*{\xtwoheadleftarrow}[2][]{%
  \ext@arrow 5097\twoheadleftarrowfill@{#1}{#2}%
}
\makeatother

\newcommand{\inj}{\hookrightarrow}

\newcommand{\op}[1]{\operatorname{{#1}}}

\newcommand{\mc}[1]{\mathcal{{#1}}}

\newcommand{\Hom}{\op{Hom}}

\newcommand{\GL}{\op{GL}}

\newcommand{\abs}[1]{\left\lvert#1\right\rvert}

\begin{document}
    \title{Torus bundles over lens spaces}
    \author{Oliver H. Wang}
	\maketitle
\begin{abstract}
Let $p$ be an odd prime and let $\rho:\Z/p\rightarrow\GL_n(\Z)$ be an action of $\Z/p$ on a lattice and let $\Gamma:=\Z^n\rtimes_{\rho}\Z/p$ be the corresponding semidirect product.
The torus bundle $M:=T^n_{\rho}\times_{\Z/p}S^{\ell}$ over the lens space $S^{\ell}/\Z/p$ has fundamental group $\Gamma$.
When $\Z/p$ fixes only the origin of $\Z^n$, Davis and L\"uck \cite{DavisLuckTorusBundles} compute the $L$-groups $L^{\langle j\rangle}_m(\Z[\Gamma])$ and the structure set $\mc{S}^{geo,s}(M)$.
In this paper, we extend these computations to all actions of $\Z/p$ on $\Z^n$.
In particular, we compute $L^{\langle j\rangle}_m(\Z[\Gamma])$ and $\mc{S}^{geo,s}(M)$ in a case where $\underline{E}\Gamma$ has a non-discrete singular set.
\end{abstract}

\tableofcontents

\section{Introduction}
In \cite{DavisLuckKTheory} and \cite{DavisLuckTorusBundles}, Davis and L\"uck study groups of the form $\Gamma=\Z^n\rtimes_\rho\Z/p$ where $p$ is an odd prime and $\rho:\Z/p\rightarrow\GL_n(\Z)$ has no nonzero fixed points.
They compute the topological $K$-theory of the real and complex group $C^*$-algebras of $\Gamma$ in \cite{DavisLuckKTheory}.
Along the way, they compute $K^*(B\Gamma)$ and several other $K$-theory groups.
Letting $T^n_{\rho}$ denote the torus with $\Z/p$-action determined by $\rho$ and letting $S^{\ell}$ denote a sphere with a free $\Z/p$-action, define $M:=T^n_\rho\times_{\Z/p}S^{\ell}$.
The manifold $M$ is a torus bundle over a lens space and the assumption that $\Z/p$ acts freely on $\Z^n\setminus\{0\}$ implies that the action of $\Z/p$ on $T^n_\rho$ has discrete fixed points.
In \cite{DavisLuckTorusBundles}, Davis and L\"uck use the computations from \cite{DavisLuckKTheory} to determine the $L$-groups of $\Z[\Gamma]$ and the structure set of $M$ in the sense of surgery theory.

For the $L$-theory computation, Davis and L\"uck use the Farrell-Jones Conjecture for $\Gamma$ to conclude $L^{\langle-\infty\rangle}_m(\Z[\Gamma])\cong H^{\Gamma}_m\left(\underline{E}\Gamma;\mathbf{L}^{\langle-\infty\rangle}_{\Z}\right)$.
Then, they compute the homology group by inverting $p$ and inverting $2$.
After inverting $p$, $H^{\Gamma}_m\left(\underline{E}\Gamma;\mathbf{L}^{\langle-\infty\rangle}_{\Z}\right)$ becomes part of a split short exact sequence.
After inverting $2$, there is an isomorphism $H^{\Gamma}_m\left(\underline{E}\Gamma;\mathbf{L}^{\langle-\infty\rangle}_{\Z}\right)\left[\frac{1}{2}\right]\cong KO^{\Gamma}_m(\underline{E}\Gamma)\left[\frac{1}{2}\right]$ and one applies the computations of \cite{DavisLuckKTheory}.

In this paper, we study the case where the action of $\Z/p$ on $\Z^n$ is not necessarily free on $\Z^n\setminus\{0\}$.
A free abelian group with a $\Z/p$-action can be written as $\Z^n=M\oplus N\oplus\Z^c$ where $M\otimes\Q\cong\Q\left(\zeta\right)^a$ and $N\otimes\Q\cong\Q[\Z/p]^b$.
Here, we use $\zeta$ to denote a primitive $p$-th root of unity.
We say that such a module is \emph{of type $(a,b,c)$} in which case $n=a(p-1)+bp+c$.
The $\Z[\Z/p]$-module $\Z^n$ will be denoted $L$.
As in \cite{DavisLuckTorusBundles} define $\Gamma:=L\rtimes_{\rho}\Z/p$ and define $M:=T^n_{\rho}\times_{\Z/p}S^{\ell}$.
The fixed points of the corresponding $\Z/p$-action on $T^n_\rho$ is a disjoint union of $a(p-1)$ many $(b+c)$-dimensional tori rather than a discrete set when $L$ is of type $(a,b,c)$.
Proving that $L_m^{\langle-\infty\rangle}(\Z[\Gamma])$ is $p$-torsion free is one of the main difficulties in applying the machinery of \cite{DavisLuckTorusBundles} to our case.
In order to do this, we invert $2$ and study topological $K$-theory.
One of our main results is
\begin{restatable*}{theorem}{KTheory}\label{thm: K(BGamma  )}
Suppose $\Gamma$ is of type $(a,b,c)$.
Then,

\begin{enumerate}
\item \label{part: diffs vanish} The differentials $d_r^{i,j}$ in the Atiyah-Hirzebruch-Serre spectral sequence for the fibration $B\Gamma \rightarrow B\Z/p$ vanish for $r\ge2$.
\item \label{part: convergence} If $b\neq0$ or $c\neq0$, then there is an isomorphism of abelian groups
\[
K^m(B\Gamma )\cong K^m\left(T^n_\rho\right)^{\Z/p}\oplus\hat{\Z}^{(p-1)p^a2^{b+c-1}}_p.
\]
If $b=c=0$, then
\[
K^m(B\Gamma )\cong\begin{cases} K^m\left(T^n_\rho\right)^{\Z/p}\oplus\hat{\Z}^{(p-1)p^a}_p& m\text{ even}\\
K^m\left(T^n_\rho\right)^{\Z/p}&m\text{ odd}\end{cases}.
\]
\end{enumerate}
\end{restatable*}

Using the techniques in \cite{DavisLuckTorusBundles}, we are then able to compute the $L$-groups of $\Z[\Gamma]$ and the structure sets of $M$.
The description of the simple structure sets in \cite{DavisLuckTorusBundles} are nice because the torsion comes from $L(\Z)$.
In our case, we will inevitably encounter $2$-torsion coming from $L_{2m}^h(\Z[\Z/p])$.
Fortunately, we are still able to obtain the integral results below (as opposed to results that only hold after inverting $2$).
This is essentially due to the splitting after inverting $p$ of the maps in Proposition \ref{prop: LES from pushout diagram} and our understanding of the Whitehead groups (in particular \cite[Theorem 1.10]{LuckRosenthal}).

In the theorems below $\mc{P}$ will denote the set of conjugacy classes of nontrivial finite subgroups of $\Gamma$ (all of which are isomorphic to $\Z/p$).
We let $N_{\Gamma}P$ denote the normalizer of $P$ in $\Gamma$ and we let $W_{\Gamma}P:=N_{\Gamma}P/P$ denote the Weyl group.
The groups $L_m^{\langle j\rangle}\left(\Z\left[\Gamma\right]\right)$ are Ranicki's surgery groups with decoration and $\mc{S}^{per,\langle j\rangle}_{n+\ell+1}\left(M\right)$ is Ranicki's algebraic structure set.
The set $\mc{S}^{geo,\langle j\rangle}_{n+\ell+1}\left(M\right)$ is the geometric structure set obtained from the surgery exact sequence by using connective $L$-theory.
We refer to Section \ref{section: structure set} for more details.

\begin{restatable*}{theorem}{Lj}\label{thm: L^j computation}
For $j=2,1,0,\cdots,-\infty$, there is an isomorphism
\[
L^{\langle j\rangle}_m(\Z[\Gamma])\cong H_m(T_{\rho}^n;L(\Z))^{\Z/p}\oplus\bigoplus_{(P)\in\mc{P}}L^{\langle j\rangle}_m(\Z[N_\Gamma P])/L^{\langle j\rangle}_m(\Z[W_\Gamma P]).
\]
\end{restatable*}

\begin{restatable*}{theorem}{Sper}\label{thm: computation of per structure set}
For $j=2,1,0,\cdots,-\infty$, there is an isomorphism
\[
\mc{S}^{per,\langle j\rangle}_{n+\ell+1}(M)\cong H_n(T_{\rho}^n;L(\Z))^{\Z/p}\oplus \bigoplus_{(P)\in\mc{P}}L^{\langle j\rangle}_{n+\ell+1}\left(\Z\left[N_\Gamma P\right]\right)/L^{\langle j\rangle}_{n+\ell+1}\left(\Z\left[W_\Gamma P\right]\right).
\]
\end{restatable*}

\begin{restatable*}{theorem}{Sgeo}\label{thm: computation of geo structure set}
For $j=2,1,0,\cdots,-\infty$, there is an isomorphism
\[
\mc{S}^{geo,\langle j\rangle}_{n+\ell+1}(M)\cong H_n(T_{\rho}^n;L(\Z)\langle1\rangle)^{\Z/p}\oplus \bigoplus_{(P)\in\mc{P}}L^{\langle j\rangle}_{n+\ell+1}\left(\Z\left[N_\Gamma P\right]\right)/L^{\langle j\rangle}_{n+\ell+1}\left(\Z\left[W_\Gamma P\right]\right).
\]
\end{restatable*}

When $j=2$ (resp. $1$), Theorem \ref{thm: computation of geo structure set} specializes to a computation of the usual simple structure set (resp. homotopy structure set) in the sense of surgery theory.

\begin{theorem}\label{thm: main theorem geo}
There are isomorphisms
\[
\mc{S}^{geo,s}(M)\cong H_n(T_{\rho}^n;L(\Z)\langle1\rangle)^{\Z/p}\oplus \bigoplus_{(P)\in\mc{P}}L^{s}_{n+\ell+1}\left(\Z\left[N_\Gamma P\right]\right)/L^{s}_{n+\ell+1}\left(\Z\left[W_\Gamma P\right]\right)
\]
and
\[
\mc{S}^{geo,h}(M)\cong H_n(T_{\rho}^n;L(\Z)\langle1\rangle)^{\Z/p}\oplus \bigoplus_{(P)\in\mc{P}}L^{h}_{n+\ell+1}\left(\Z\left[N_\Gamma P\right]\right)/L^{h}_{n+\ell+1}\left(\Z\left[W_\Gamma P\right]\right).
\]
\end{theorem}

The $L$-groups appearing in these computations are computable; the normalizers are isomorphic to $\Z^{b+c}\times\Z/p$ and the Weyl groups are isomorphic to $\Z^{b+c}$ so Shaneson splitting allows us describe these groups in terms of the $L$-groups of $\Z$ and $\Z[\Z/p]$.

\begin{remark}
The isomorphisms in the theorems above do not come from ``natural'' maps.
For instance, the map $N_\Gamma P\rightarrow\Gamma$ induces a map $L^{\langle j\rangle}_m\left(\Z\left[N_\Gamma P\right]\right)/L^{\langle j\rangle}_m\left(\Z\left[W_\Gamma P\right]\right)\rightarrow L^{\langle j\rangle}_m\left(\Z\left[\Gamma\right]\right)$.
After composing with the isomorphism in Theorem \ref{thm: L^j computation}, the image is a $p$-power index subgroup of $L^{\langle j\rangle}_m\left(\Z\left[N_\Gamma P\right]\right)/L^{\langle j\rangle}_m\left(\Z\left[W_\Gamma P\right]\right)$.
\end{remark}

\begin{remark}
To prove Theorem \ref{thm: L^j computation}, it suffices to prove the case that $\Gamma$ is of type $(a,b,0)$ i.e. $\Gamma=L\rtimes\Z/p$ where $L$ is a free abelian group with no $\Z$-summands with trivial $\Z/p$ action.
Indeed, one can inductively apply Shaneson splitting to obtain the more general case.
It would be convenient to do this with the structure sets as well.
However, we are not aware of a reference that gives Shaneson splitting for the structure sets.
\end{remark}

\subsection{Geometric Interpretations of the Structure Set}
In the situation of \cite{DavisLuckTorusBundles}, the computation of the structure set is interpreted as follows.
Suppose $f:N\rightarrow M$ is a structure and let $\overline{f}:\overline{N}\rightarrow\overline{M}\cong T^n\times S^\ell$ denote the $\Z/p$-cover.
The proof of \cite[Theorem 10.6]{DavisLuckTorusBundles} applies in our case so we have the following interpretation of the $H_n(T_{\rho}^n;L(\Z)\langle1\rangle)^{\Z/p}$ summand of the structure set.
\begin{theorem}\label{thm: splitting invariants}
The following are equivalent:
\begin{enumerate}
\item The structure $[f:N\rightarrow M]\in\mc{S}^{geo,s}(M)$ vanishes under projection to $H_n(T_{\rho}^n;L(\Z)\langle1\rangle)^{\Z/p}$;
\item $\overline{f}:\overline{N}\rightarrow\overline{M}$ is homotopic to a homeomorphism;
\item For a nonempty $J\subseteq\{1,\cdots,n\}$, let $T^J\subseteq T^n$ denote the obvious subtorus.
After making $\overline{f}$ transverse to $T^J\times\{pt\}\subseteq T^n\times S^\ell$, we obtain a surgery problem
\[
\overline{f}:\overline{f}^{-1}(T^J\times\{pt\})\rightarrow T^J\times\{pt\}.
\]
This has a vanishing surgery obstruction in $L_{\abs{J}}\left(\Z\left[\Z^{\abs{J}}\right]\right)$ for all nonempty $J\subseteq\{1,\cdots,n\}$.
\end{enumerate}
\end{theorem}

Understanding the $\bigoplus_{(P)\in\mc{P}}L^{s}_{n+\ell+1}\left(\Z\left[N_\Gamma P\right]\right)/L^{s}_{n+\ell+1}\left(\Z\left[W_\Gamma P\right]\right)$ summand is more difficult.
Shaneson splitting implies that $L^h_{2m}\left(\Z\left[\Z/p\right]\right)$ will appear as summands of $L^s_{n+\ell+1}\left(\Z\left[N_\Gamma P\right]\right)$.
These groups have $2$-torsion which involves the ideal class groups.
Rationally, there is an isomorphism
\[
\bigoplus_{(P)\in\mc{P}}L^{s}_{n+\ell+1}\left(\Z\left[N_\Gamma P\right]\right)/L^{s}_{n+\ell+1}\left(\Z\left[W_\Gamma P\right]\right)\otimes\Q\cong\bigoplus_{(P)\in\mc{P}}\bigoplus_{\substack{k\le b+c\\k+\ell+1\text{ even}}}\left(R_{\C}^{(-1)^{k+\ell+1/2}}(\Z/p)/\langle\op{reg}\rangle\right)^{\binom{b+c}{k}}\otimes\Q.
\]
In the above expression, $R_{\C}^{\pm}(\Z/p)$ denotes the group of virtual complex $\Z/p$-representations whose characters are of the form $\chi\pm\chi^{-1}$ and $\op{reg}$ denotes the regular representation.

We give a heuristic description of this summand in the case $\Gamma=\Z^n\times \Z/p$ (i.e. when $M$ is the product of a torus and a lens space $L^\ell$).
Note that the inner sum on the right hand side is indexed by the standard subtori $T^k\subseteq T^n$ where $k+\ell+1$ is even.
Suppose $f:N\rightarrow M$ is a structure and let $T^k\subseteq T^n$ be a subtorus such that $k+\ell+1$ is even.
Suppose $f$ is transverse to the submanifold $T^k\times L^\ell$.
Then $f^{-1}\left(T^k\times L^\ell\right)$ has a $\Z/p$-cover so we may take its $\rho$-invariant to obtain an element of $\left(R_{\C}^{(-1)^{k+\ell+1/2}}(\Z/p)/\langle\op{reg}\rangle\right)\otimes\Q$.
This gives an element in the summand corresponding to the torus $T^k$.
This description is not technically correct; without some modifications, we do not know how to show it is well-defined.
A rigorous interpretation of these $\rho$-invariants will be the subject of future work.

\subsection{Outline}
In Section \ref{section: group theory}, we review properties of the $\Z[\Z/p]$-module $L$ and we state relevant properties of the group $\Gamma$.
In section \ref{section: equiv homology and fjc}, we introduce some machinery from \cite{DavisLuckTorusBundles}; our $L$-theory computations rely on the Farrell-Jones Conjecture and a description of $\underline{E}\Gamma$ as a homotopy pushout.
Section \ref{section: K theory} is the main computational part of the paper.
It is devoted to computing the topological $K$-theory of $B\Gamma$ and recording some consequences of the computation.
The main computational tool we use is the Atiyah-Hirzebruch-Serre spectral sequence.
In sections \ref{section: L theory} and \ref{section: structure set}, we compute the $L$-groups of $\Z[\Gamma]$ and the structure set of $M$.
These computations follow the computations in \cite{DavisLuckTorusBundles} very closely.
For section \ref{section: structure set}, in particular, our results follow from the proofs of \cite{DavisLuckTorusBundles} with only slight modifications.

\subsection{Acknowledgments}
The author would like to thank Shmuel Weinberger for suggesting this project and for many helpful discussions.
The author would also like to thank the referee for their detailed report.

\section{Group Theoretic Preliminaries}\label{section: group theory}
We are interested in groups $\Gamma$ of the form $L\rtimes\Z/p$ where $L$ is a finitely generated free abelian group.
Throughout this paper, $p$ will always be an odd prime.
Curtis and Reiner classified these groups in \cite[Theorem 74.3]{CurtisReiner}.

\begin{theorem}\label{thm: integral reps of Z/p}
Let $\zeta$ be a primitive $p$-th root of unity.
If $L$ is an indecomposable integral $\Z/p$ representation then $L$ is either
\begin{enumerate}
    \item \label{type 2 module} $B\subseteq\Q(\zeta)$ a fractional ideal with action given by multiplication by $\zeta$.
    \item \label{type 3 module} $B\oplus\Z$ where $B\subseteq\Q(\zeta)$ is a fractional ideal and $t\cdot(b,m)=(\zeta b+mb_0,m)$ where $b_0\in B\setminus(1-\zeta)B$.
    We will denote this by $B\oplus_{b_0}\Z$.
    Two such representations of this form are isomorphic if the fractional ideals represent the same element in the ideal class group.
    \item $\Z$ with a trivial action.
\end{enumerate}
\end{theorem}

\begin{example}
Let $\zeta$ be a $p$-th root of unity.
Then $\Z[\zeta]$ is an example of a module of the form \ref{type 2 module} in Theorem \ref{thm: integral reps of Z/p}.
The group ring $\Z[\Z/p]$ is an example of a module of the form \ref{type 3 module}; it is of the form $\Z\left[\zeta\right]\oplus_1\Z$.
When $p=3$, for instance, the assignment
\[
\left(a_0+a_1\zeta,m\right)\mapsto a_0\left(-a_0+a_1+m\right)t+\left(-a_1+m\right)t^2
\]
defines an isomorphism of $\Z\left[\Z/p\right]$-modules.
\end{example}

The group $\Gamma$ determines a torus bundle $M$ over a lens space with fundamental group $\Z/p$.
Note that the manifold corresponding to $(L\oplus\Z)\rtimes\Z/p$ is a product $M\times S^1$.
As taking products with $S^1$ is a well understood operation in topology, we will sometimes make the simplification that $L$ does not have any $\Z$ summands.

\begin{definition}\label{def: type of Z/p module}
We say that a $\Z/p$-module $L$ is \emph{of type $(a,b,c)$} if it is of the form $L=\left(\bigoplus_{i=1}^a M_i\right)\oplus\left(\bigoplus_{j=1}^b N_j\right)\oplus\Z^c$ where $M_i$ is in form (\ref{type 2 module}) and $N_j$ is of the form (\ref{type 3 module}) in Theorem \ref{thm: integral reps of Z/p}.
We say that $L\rtimes\Z/p$ is \emph{of type $(a,b,c)$} if $L$ is of type $(a,b,c)$.
\end{definition}

\begin{lemma}\label{lem: properties of Gamma }
Suppose $\Gamma=L\rtimes\Z/p$ is of type $(a,b,c)$.
Then the following hold.
\begin{enumerate}
\item The virtually cyclic subgroups of $\Gamma$ are isomorphic to either the trivial group, $\Z/p$, $\Z$ or $\Z\times\Z/p$.
\item Let $\mc{P}$ denote the set of conjugacy classes of maximal finite subgroups of $\Gamma$.
Then, $\abs{\mc{P}}=p^a$.
\item If $P$ is a finite subgroup of order $p$, then the normalizer $N_\Gamma P$ is isomorphic to $\Z^{b+c}\times P$ and the Weyl group $W_\Gamma P:=N_\Gamma P/P$ is isomorphic to $\Z^{b+c}$.
\end{enumerate}
\end{lemma}
\begin{proof}
The elements of our group can be written as $xy^i$ where $x\in L$ and $y$ is a generator of $\Z/p$.
Let $\rho(-):L\rightarrow L$ denote the action of $\Z/p$ on $L$.

To show the first statement, if $xy$ and $x'y$ are in a subgroup $H$, then $x(x')^{-1}$ must be in $H$.
Hence, if $H$ is finite, $x=x'$.
It follows that the nontrivial finite subgroups are isomorphic to $\Z/p$.
Suppose $V$ is an infinite virtually cyclic subgroup that is not infinite cyclic.
Then $V$ must surject onto $\Z$ with kernel $\Z/p$ and $V\cap L$ is an infinite cyclic group.
Let $xy$ be a torsion element of $V$ and let $v\in V\cap L$.
Then $xyvy^{-1}x^{-1}=\rho(v)\in V\cap L$.
But if $\rho(v)\neq v$, $V\cap L$ would contain a subgroup isomorphic to $\Z^2$.
Therefore, $V\cap L$ is fixed by $\rho$.
It follows that $V$ is isomorphic $\Z\times\Z/p$.

For the second statement, observe that $(xy)^p=\sum_{i=0}^{p-1}y^i\cdot x$.
Therefore, $xy$ is torsion if and only if $x$ is in the kernel of the norm map $\text{Norm}:L\rightarrow L$.
Moreover, if $x$ and $z$ are in the kernel of the norm map, one checks that the group generated by $xy$ is conjugate to the group generated by $zy$ if an only if $x-z$ is in the image of $1-y:L\rightarrow L$.
Therefore, $\mc{P}$ is in bijection with $H^1(\Z/p;L)$.
It follows from Proposition \ref{prop: simplifying Tate cohom}, which we prove later, that this is isomorphic to $H^1\left(\Z/p;\Z[\zeta]^a\right)$.
By \cite[Lemma 1.10(i)]{DavisLuckKTheory}, this group is isomorphic to $(\Z/p)^a$.

The third statement follows from the beginning of the proof of \cite[Theorem 1.10]{LuckRosenthal}.
\end{proof}

\section{Equivariant Homology and the Farrell-Jones Conjecture}\label{section: equiv homology and fjc}
In this section, we introduce some preliminary material on equivariant homology following \cite{DavisLuckSpacesOverACategory} and on the Farrell-Jones Conjecture, which will allow us to compute $L$ and $K$-groups.

\subsection{Equivariant Homology}
Let $G$ be a discrete group.
Given a covariant functor $\mathbf{E}:Grpd\rightarrow Sp$ from the category of small groupoids to spectra, define the equivariant homology groups of a $G$-CW-complex $X$ to be
\[
H_m^G(X;\mathbf{E}):=\pi_m\left(X^-_+\wedge_{\op{Or}(G)}\mathbf{E}(\overline{G/-})\right)
\]
where $\op{Or}(G)$ is the orbit category of $G$ and $\overline{G/H}$ is the groupoid associated to the $G$-set $G/H$ and $-\wedge_{\op{Or}(G)}-$ denotes a coend.
The functor $X^-_+$ sends an orbit $G/H$ to the fixed point set $X^H_+$.
If $f:X\rightarrow Y$ is a map of $G$-CW-complexes, then we write
\[
H_m^G\left(X\rightarrow Y;\mathbf{E}\right):=H_m^G\left(\op{cone}(f);\mathbf{E}\right).
\]

In this paper, we will take $\mathbf{E}$ to be $\mathbf{K}_{\Z}^{-\infty}$ and $\mathbf{L}_{\Z}^{\langle j\rangle}$ for $j=2,1,0,\cdots,-\infty$.
The corresponding homology theories have the property that $H_m^G\left(G/H;\mathbf{K}_{\Z}^{-\infty}\right)\cong K_m(\Z[H])$ and $H_m^G\left(G/H;\mathbf{L}_{\Z}^{\langle j\rangle}\right)\cong L^{\langle j\rangle}_m(\Z[H])$ for all $m\in\Z$.
We also use equivariant topological $K$-theory, which sends $G/H$ to the representation ring $R_{\C}(H)$ when $G$ is finite.

Equivariant cohomology is defined analogously.
We refer to \cite{DavisLuckSpacesOverACategory} for more details.

\begin{remark}
The notation $H_m^G\left(X;\mathbf{L}^{\langle j\rangle}_{\Z}\right)$ denotes the Davis-L\"uck equivariant homology as mentioned above whereas the notation $H_m\left(X;L^{\langle j\rangle}\left(\Z\right)\right)$ denotes the generalized homology of $X$ with coefficients in the spectrum $L^{\langle j\rangle}\left(\Z\right)$.
\end{remark}

\subsection{Classifying Spaces}
\begin{definition}\label{def: classifying space for family}
Let $G$ be a group.
A \emph{family} of subgroups is a nonempty set $\mc{F}$ of subgroups closed under taking subgroups and conjugation.
A \emph{classifying space for $\mc{F}$}, denoted $E_{\mc{F}}G$, is a $G$-CW-complex satisfying
\[
\left(E_{\mc{F}}G\right)^H\simeq\begin{cases}
pt&H\in\mc{F}\\
\emptyset&H\notin\mc{F}
\end{cases}.
\]
\end{definition}

\begin{example}
If $\{e\}$ is the family consisting of only the trivial group, then $E_{\{e\}}G=EG$.
The primary families we will consider are $\mc{V}cyc$, the collection of virtually cyclic subgroups, and $\mc{F}in$, the collection of finite subgroups.
We will use the following notation.
\[
\underline{\underline{E}}G:=E_{\mc{V}cyc}G\hspace{1cm}\underline{E}G:=E_{\mc{F}in}G
\]
\end{example}

Specifying to the case where $\Gamma =L\rtimes\Z/p$, \cite[Corollary 2.10]{LuckWeiermann} shows that there is the following homotopy pushout diagram.
\begin{equation}\label{diagram: pushout for EGamma  }
\begin{tikzpicture}[baseline=(current  bounding  box.center), scale =1.5]
\node (A) at (0,1) {$\coprod_{(P)\in\mc{P}}\Gamma \times_{N_\Gamma P}EN_\Gamma P$}; \node (B) at (2,1) {$E\Gamma $};
\node (C) at (0,0) {$\coprod_{(P)\in\mc{P}}\Gamma \times_{N_\Gamma P}EW_\Gamma P$}; \node (D) at (2,0) {$\underline{E}\Gamma $};
\path[->] (A) edge (B) (A) edge (C) (B) edge (D) (C) edge (D);
\end{tikzpicture}
\end{equation}
Proceeding as in \cite[Lemma 7.2]{DavisLuckKTheory} and using that $EW_\Gamma P$ is a model of $\underline{E}N_\Gamma P$ as an $N_\Gamma P$-space, we obtain the following long exact sequences.
\begin{prop}\label{prop: LES from pushout diagram}
For an equivariant homology theory $\mc{H}^\Gamma _{\bullet}(-)$, there is a long exact sequence
\[
\cdots\rightarrow\mc{K}_m\xrightarrow{\phii_m}\mc{H}_m^\Gamma (\underline{E}\Gamma )\xrightarrow{\op{ind}_{\Gamma \rightarrow 1}}\mc{H}_m(\underline{B}\Gamma )\rightarrow\mc{K}_{m-1}\rightarrow
\]
where $\mc{K}_m:=\bigoplus_{(P)\in\mc{P}}\ker\left(\mc{H}_m^{N_\Gamma P}(\underline{E}N_\Gamma P)\rightarrow\mc{H}_m(\underline{B}N_\Gamma P)\right)$.
After inverting $p$, the map $\mc{H}_m^\Gamma (\underline{E}\Gamma )\rightarrow\mc{H}_m(\underline{B}\Gamma )$ is a split surjection.

For an equivariant cohomology theory $\mc{H}_\Gamma ^{\bullet}(-)$, there is a long exact sequence
\[
\cdots\rightarrow\mc{C}^{m-1}\rightarrow\mc{H}^m(\underline{B}\Gamma )\xrightarrow{\op{ind}_{\Gamma \rightarrow 1}}\mc{H}^m_\Gamma (\underline{E}\Gamma )\xrightarrow{\phii_m}\mc{C}^m\rightarrow\cdots
\]
where $\mc{C}^m:=\bigoplus_{(P)\in\mc{P}}\op{coker}\left(\mc{H}^m(\underline{B}N_\Gamma P)\rightarrow\mc{H}^m_{N_{\Gamma P}}(\underline{E}N_\Gamma P)\right)$.
After inverting $p$, the map $\mc{H}^m(\underline{B}\Gamma )\rightarrow\mc{H}^m_\Gamma (\underline{E}\Gamma )$ is a split injection.

The maps $\phii_m$ are induced by the inclusions $P\rightarrow \Gamma $.
\end{prop}

\begin{remark}\label{remark: mc K in sequence}
Since $N_\Gamma P\cong\Z^{b+c}\times\Z/p$, we have $\underline{E}N_\Gamma P\simeq\R^{b+c}$ with a trivial $\Z/p$-action.
In particular, $\underline{B}N_\Gamma P$ is just $T^{b+c}$.
For the homology theory $H_{\bullet}^{\Gamma}\left(-;\mathbf{L}_{\Z}^{\langle j\rangle}\right)$ we can make the identifications
\begin{align*}
H_m^{\Z^{b+c}\times\Z/p}\left(E\Z^{b+c};\mathbf{L}^{\langle j\rangle}_\Z\right)&\cong H_m^{\Z/p}\left(T^{b+c};\mathbf{L}^{\langle j\rangle}_{\Z}\right)\cong H_m\left(T^{b+c};L^{\langle j\rangle}\left(\Z[\Z/p]\right)\right)\\
H_m^{\Z^{b+c}}\left(E\Z^{b+c};\mathbf{L}^{\langle j\rangle}_\Z\right)&\cong H_m\left(T^{b+c};L^{\langle j\rangle}(\Z)\right)
\end{align*}
and we can describe $\mc{K}_m$ as
\[
\bigoplus_{(P)\in\mc{P}} H_m\left(T^{b+c};\tilde{L}^{\langle j\rangle}\left(\Z[\Z/p]\right)\right).
\]
Here, the spectrum $\tilde{L}^{\langle j\rangle}\left(\Z\left[\Z/p\right]\right)$ is the cofiber of the map $L^{\langle j\rangle}\left(\Z\right)\rightarrow L^{\langle j\rangle}\left(\Z\left[\Z/p\right]\right)$.
\end{remark}

\subsection{The Farrell-Jones Conjecture}
One of the primary computational tools that we use is the Farrell-Jones Conjecture, which has been proved in many cases.
\cite{BartelsLuckCAT0} proves the conjecture for our group $\Gamma $.

\begin{theorem}\label{thm: FJC for Gamma }
The map $\underline{\underline{E}}\Gamma \rightarrow pt$ induces isomorphisms on $H^\Gamma _{\bullet}\left(-;\mathbf{K}^{-\infty}_{\Z}\right)$ and $H^\Gamma _{\bullet}\left(-;\mathbf{L}^{\langle-\infty\rangle}_{\Z}\right)$.
\end{theorem}

By Lemma \ref{lem: properties of Gamma  }, \cite[Theorem 65]{LuckReich} and \cite[Proposition 75]{LuckReich}, we obtain the following.

\begin{prop}\label{prop: transitivity}
The map $\underline{E}\Gamma \rightarrow pt$ induces isomorphisms on $\mathbf{L}^{\langle-\infty\rangle}_{\Z}$-homology.
Hence,
\[
H_*\left(\underline{E}\Gamma;\mathbf{L}^{\langle-\infty\rangle}_\Z\right)\cong L^{\langle-\infty\rangle}_*\left(\Z\left[\Gamma\right]\right).
\]
\end{prop}

\section{Topological $K$-Theory}\label{section: K theory}
This is the main computation section of the paper.
The goal of this section is to prove the following theorem.

\KTheory

In Theorem \ref{thm: K(BGamma  )} we use $\hat{\Z}_p$ to denote the $p$-adic integers.
We will reduce to the case $c=0$ and proceed by induction on $b$.
The case where $L$ is type $(a,0,0)$ is \cite[Theorem 3.1]{DavisLuckKTheory}.

\subsection{Some Preliminaries}
\subsubsection{Group Cohomology}
We collect some important facts about group cohomology.
For a finite group $G$ and a $\Z[G]$-module $M$, let $\hat{H}^*(G;M)$ denote the Tate cohomology.
For $i\ge1$, $\hat{H}^i(G;M)=H^i(G;M)$ and, for $i\le-2$, $\hat{H}^i(G;M)=H_{-i-1}(G;M)$.
There is an exact sequence
\[
0\rightarrow\hat{H}^{-1}(G;M)\rightarrow M_G\xrightarrow{\text{Norm}}M^G\rightarrow\hat{H}^0(G;M)\rightarrow0
\]
where $\text{Norm}(x)=\sum_{g\in G}g\cdot x$.

Let $M^*$ denote the dual $\Hom_{\Z}(M,\Z)$.
This has a $G$-action via $(g\cdot f)(x)=f(g^{-1}\cdot x)$.
The following is \cite[Lemma A.1]{DavisLuckKTheory}.
\begin{lemma}\label{lem: Tate duality}
Let $G$ be a finite group and let $M$ be a finitely generated $\Z[G]$ module with no $p$-torsion for all primes $p$ dividing the order of $G$.
Then, for all $i\in\Z$, there is an isomorphism
\[
\hat{H}^i(G;M)\cong\hat{H}^{-i}\left(G;M^*\right).
\]
\end{lemma}

\begin{prop}\label{prop: simplifying Tate cohom}
Suppose $L$ is a module of type $(a,b,0)$.
Then there is an isomorphism of Tate cohomology groups
\[
\hat{H}^*\left(\Z/p;\Lambda^r L\right)\cong\hat{H}^*\left(\Z/p;\Lambda^r\left(\Z[\zeta]^a\oplus\Z[\Z/p]^b\right)\right).
\]
\end{prop}
\begin{proof}
We follow the proof of \cite[Lemma 1.10(i)]{DavisLuckKTheory} where the case $b=0$ is done.
Since $\hat{H}(\Z/p;M)\cong\hat{H}(\Z/p;M\otimes \Z_{(p)})$, it suffices to check that $\Lambda^r L\otimes\Z_{(p)}\cong\Lambda^r\left(\Z[\zeta]^a\oplus\Z[\Z/p]^b\right)\otimes\Z_{(p)}$.
As it has been shown in the proof of \cite[Lemma 1.10(i)]{DavisLuckKTheory} that $L\otimes\Z_{(p)}\cong \Z[\zeta]^a\otimes\Z_{(p)}$ for $L$ of type $(a,0,0)$, it suffices to show that $L\otimes\Z_{(p)}\cong \Z[\Z/p]^b\otimes\Z_{(p)}$ for $L$ of type $(0,b,0)$.

We may assume $b=1$ so $L=B_0\oplus_{b_0}\Z$ and $\Z[\Z/p]=B_1\oplus_{b_1}\Z$ where $B_0$ and $B_1$ are fractional ideals of $\Z[\zeta]$ and $b_i\notin (1-\zeta)B_i$.
As $\Z_{(p)}[\zeta]$ is a PID, its ideal class group is trivial.
Hence, any fractional ideal $I$ of $\Z_{(p)}[\zeta]$ is of the form $\alpha\Z_{(p)}[\zeta]\subseteq\Q(\zeta)$ for some $\alpha\in\Q(\zeta)$.
In particular, there are numbers $\alpha_0,\alpha_1\in\Z_{(p)}[\zeta]$ such that $\alpha_0B_0\otimes\Z_{(p)}=\alpha_1B_1\otimes\Z_{(p)}$.
The $\Z[\Z/p]$-module structures of $L$ and $\Z[\Z/p]$ do not change if we change the choice of $b_0$ and $b_1$ so long as they remain outside $(1-\zeta)B_i$.
By multiplying $b_0$ with some integer prime to $p$, we may assume $b_0\in B_0\setminus(1-\zeta)B_0$ and $\alpha_1^{-1}\alpha_0b_0\in B_1\setminus(1-\zeta)B_1$.
Let $b_1=\alpha_1^{-1}\alpha_0b_0$.
Then,
\begin{align*}
\left(B_0\oplus_{b_0}\Z\right)\otimes\Z_{(p)}&\cong\left(B_0\otimes\Z_{(p)}\right)\oplus_{b_0}\Z_{(p)}\\
&\cong\left(\alpha_0B_0\otimes\Z_{(p)}\right)\oplus_{\alpha_0 b_0}\Z_{(p)}\\
&\cong\left(\alpha_1B_1\otimes\Z_{(p)}\right)\oplus_{\alpha_1 b_1}\Z_{(p)}\\
&\cong\left(B_1\otimes\Z_{(p)}\right)\oplus_{b_1}\Z_{(p)}\cong\left(B_1\oplus_{b_1}\Z\right)\otimes\Z_{(p)}.
\end{align*}
Here, the second line is obtained by the isomomorphism $(b,m)\mapsto(\alpha_0b,m)$ and the fourth line follows from a similar isomorphism.
The third line follows from the fact that $\alpha_0B_0\otimes\Z_{(p)}\cong\alpha_1B_1\otimes\Z_{(p)}$ and from our choice of $b_1$.
\end{proof}

The following is \cite[Lemma 1.10]{DavisLuckKTheory}.

\begin{prop}\label{prop: Davis-Luck group cohom}
Suppose $L$ is a module of type $(a,0,0)$.
Then
\[
\hat{H}^i(\Z/p;H^j(L))\cong\bigoplus_{\substack{\ell_1+\cdots+\ell_k=j\\0\le\ell_q\le p-1}}\hat{H}^{i+j}(\Z/p;\Z)\cong\begin{cases}(\Z/p)^{a_j}&i+j\text{ even}\\0&i+j\text{ odd}\end{cases}
\]
where $a_j$ is the number of partitions of $j$.
\end{prop}

We will also need the following lemma, which appears in the proof of \cite[Lemma 1.10]{DavisLuckKTheory}.

\begin{lemma}\label{lem: exterior is usually free}
For $1\le m\le p-1$, $\Lambda^m\Z[\Z/p]$ is free as a $\Z[\Z/p]$-module.
\end{lemma}

The following result computes the fixed sets of the $\Z/p$-action on the torus corresponding to a module of type $(a,b,c)$.

\begin{prop}\label{prop: fixed sets of torus}
Suppose $L$ is a module of type $(a,b,c)$.
Let $T^n_{\rho}$ denote the torus $(L\otimes\R)/L$.
Then $(T^n_{\rho})^{\Z/p}\cong \left(T^{(b+c)}\right)^{\coprod a(p-1)}$.
\end{prop}
\begin{proof}
In the case $L$ is of type $(a,0,0)$, this is \cite[Lemma 1.9(v)]{DavisLuckKTheory}.
The case where $L$ is of type $(0,0,c)$ is straightforward.
Since $T^n_{\rho}$ is equivariantly a product, it suffices to show that, when $L$ is of type $(0,1,0)$, the fixed set is a circle.

Suppose $L$ is of type $(0,1,0)$.
Consider the following short exact sequence of $\Z/p$-modules.
\[
L\rightarrow L\otimes\R\rightarrow T^p_{\rho}
\]
This gives rise to the top exact sequence in the diagram below.
\[
\begin{tikzpicture}[scale=2]
\node (A) at (0,1) {$L^{\Z/p}$};\node (B) at (2,1) {$(L\otimes\R)^{\Z/p}$};\node (C) at (4,1) {$(T^p_{\rho})^{\Z/p}$}; \node (D) at (6,1) {$H^1(\Z/p;L)$};
\node (E) at (0,0) {$\Z[\Z/p]^{\Z/p}$};\node (F) at (2,0) {$\R[\Z/p]^{\Z/p}$};\node (G) at (4,0) {$S^1$};\node (H) at (6,0) {$0$};
\path[->] (A) edge (B) (B) edge (C) (C) edge (D) (A) edge (E) (B) edge node[left]{$\cong$} (F) (C) edge (G) (D) edge (H) (E) edge (F) (F) edge (G) (G) edge (H);
\end{tikzpicture}
\]
By the proof of Proposition \ref{prop: simplifying Tate cohom}, there is an isomorphism of $\Z/p$-modules $L\otimes\R\cong\R[\Z/p]$ which identifies $L$ with a finite index $\Z/p$-submodule of $\Z[\Z/p]$.
This gives the vertical maps.
This also implies that $H^1(\Z/p;L)=0$.
It follows that $(T^p_{\rho})^{\Z/p}\cong S^1$.
\end{proof}

\subsubsection{Cohomology of $T^n_{\rho}$}
In order to relate these algebraic results to the problem of computing topological $K$-theory, we record some results on the cohomology of $T^n_{\rho}$ as a $\Z[\Z/p]$-module.
Let $L$ be a module of type $(a,b,c)$ determining the representation $\rho$.
Then, as $\Z[\Z/p]$-modules, $H_1(T_{\rho}^n)\cong L$ and $H^1(T_{\rho}^n)\cong L^*$.
Moreover, $H^*(T_{\rho}^n)\cong\Lambda^* L^*$.
It is clear that $L^*$ is also a module of type $(a,b,c)$.

We will need to use the topological $K$-theory of $T^n_{\rho}$ considered as a $\Z[\Z/p]$-module.
The Atiyah-Hirzebruch spectral sequence collapses for tori so, as an abelian group, $K^m(T^n)\cong\bigoplus_{\ell\in\Z}H^{m+2\ell}(T^n)$.
The proof of \cite[Lemma 3.3]{DavisLuckKTheory} shows this is also true as $\Z[\Z/p]$-modules.
\begin{lemma}\label{lem: K(T) as Z/p mod}
Let $T_{\rho}^n$ be a torus with a $\Z/p$-action as above.
Then as a $\Z[\Z/p]$-module,
\[
K^m(T_{\rho}^n)\cong\bigoplus_{\ell\in\Z}H^{m+2\ell}(T^n).
\]
\end{lemma}

\subsubsection{Facts about Spectral Sequences}

Suppose $E\rightarrow B$ is a fibration with connected base space and with fiber $F$.
Let $\mc{H}^*$ be a generalized cohomology theory and let $\mc{H}_*$ be a generalized homology theory.
There are Atiyah-Hirzebruch-Serre spectral sequences
\begin{align}\label{ss: ahs ss}
E_2^{i,j}&=H^i(B;\mc{H}^j(F))\Rightarrow\mc{H}^{i+j}(E)\\
E^2_{i,j}&=H_i(B;\mc{H}_j(F))\Rightarrow\mc{H}_{i+j}(E).
\end{align}

In the cohomology spectral sequence above, we have
\[
E_2^{0,j}=H^0(G;\mc{H}^j(F))=\mc{H}^j(F)^{G}
\]
when $B$ is path connected with fundamental group $G$.
Thus, $E_\infty^{0,j}$ is a subgroup of $\mc{H}^j(F)^{G}$.
In the homology spectral sequence,
\[
E^2_{0,j}=H_0(G;\mc{H}_j(F))=\mc{H}_j(F)_G
\]
so $E^\infty_{0,j}$ is a quotient of $\mc{H}_j(F)_G$.

The following is in the appendix of \cite{DavisLuckKTheory}.
\begin{theorem}
The composite $\mc{H}^j(E)\rightarrow E_\infty^{0,j}\inj\mc{H}^j(F)^{G}$ is equal to the map on cohomology induced by the inclusion $F\inj E$.
In particular, $\mc{H}^j(E)\rightarrow\mc{H}^j(F)^{G}$ is surjective if and only if the differentials $d_r^{0,j}$ vanish for $r\ge2$.

The composite $\mc{H}_j(F)_G\rightarrow E^\infty_{0,j}\rightarrow\mc{H}_j(E)$ is equal to the map on homology induced by the inclusion $F\inj E$.
In particular, $\mc{H}_j(F)_G\rightarrow\mc{H}_j(E)$ is injective if and only if the differentials $d^r_{r,j-r-1}$ vanish for $r\ge2$.
\end{theorem}

The fibrations we use will come from group extensions $N\rightarrow G\rightarrow G/N$ where $G/N$ is finite.
We have the inclusion of $BN$ as the fiber of $BG\rightarrow B(G/N)$.
This map induces the inclusion $N\inj G$ so, up to homotopy we may think of $BN\rightarrow BG$ as a covering space with fibers $G/N$.
If $G/N$ is finite, then there is a transfer map $\tau^*:\mc{H}^m(BN)\rightarrow\mc{H}^m(BG)$ such that the composition
\[
\mc{H}^m(BN)\xrightarrow{\tau^*}\mc{H}^m(BG)\rightarrow\mc{H}^m(BN)^{G/N}
\]
is the norm map.

For a generalized homology theory there is a transfer map $\tau_*:\mc{H}_m(BG)\rightarrow\mc{H}_m(BN)$ such that the composition
\[
\mc{H}_m(BN)_{G/N}\rightarrow\mc{H}_m(BG)\xrightarrow{\tau_*}\mc{H}_m(BN)
\]
is the norm map.
To summarize,

\begin{prop}\label{prop: norms and E_infty}
Suppose $N\rightarrow G\rightarrow G/N$ is a group extension with $G/N$ finite.
In the cohomological Atiyah-Hirzebruch-Serre spectral sequence for the fibration $BN\rightarrow BG\rightarrow B(G/N)$, an element $x\in E_2^{0,j}$ is nonzero in $E_{\infty}^{0,j}$ if it is in the image of the norm map.
In the homological spectral sequence, an element $x\in E^2_{0,j}$ represents a nonzero element in $E^{\infty}_{0,j}$ if the norm of $x$ is nonzero.
\end{prop}

We specialize \cite[Theorem 13.2]{Boardman} to the following statement.
\begin{theorem}\label{thm: boardman convergence}
If there is an $N>0$ such that the differentials $d_r^{i,j}$ in the spectral sequence \refeq{ss: ahs ss} vanish for $r>N$, then this converges strongly in the following sense.
For the filtration
\[
\cdots F^s\subseteq F^{s-1}\subseteq\cdots\subseteq F^1\subseteq F^0=\mc{H}^{m}(X),
\]
the following hold:
\begin{enumerate}
    \item $\bigcap_{s=0}^\infty F^s=0$;
    \item $\mc{H}^{m}(X)=\varprojlim_s \mc{H}^{m}(X)/F^s$.
\end{enumerate}
\end{theorem}

\subsection{Vanishing of Differentials}
We now prove Theorem \ref{thm: K(BGamma  )}(\ref{part: diffs vanish}).
First, we reduce to the case where $\Gamma$ is of type $(a,b,0)$.

\begin{lemma}\label{lem: reduction of diffs vanishing to (a,b,0)}
Suppose \ref{thm: K(BGamma  )}(\ref{part: diffs vanish}) is true for groups $\Gamma$ of type $(a,b,c)$.
Then it is true for groups of type $(a,b,c+1)$.
\end{lemma}
\begin{proof}
If $\Gamma$ is of type $(a,b,c+1)$, then it is isomorphic to $\Gamma'\times\Z$ where $\Gamma'$ is of type $(a,b,c)$.
Since $B\Gamma\cong S^1\times B\Gamma'$, there is a morphism of spectral sequences coming from the map of fibrations $B\Gamma'\rightarrow B\Gamma$ over $B\Z/p$.

Let $B\Z/p^{(s)}$ denote the $s$-skeleton of $B\Z/p$.
Let $Y_s$ denote the preimage of $B\Z/p^{(s)}$ under the map $B\Gamma'\rightarrow B\Z/p$.

The exact couple giving rise to the spectral sequence for $B\Gamma'$ is given by the abelian groups $A_1^{s,m-s}=K^m(Y_s)$ and $E_1^{s,m-s}=K^m(Y_s,Y_{s-1})$ with maps
\[
K^{m-1}(Y_s)\rightarrow K^m(Y_{s+1},Y_s)\rightarrow K^m(Y_{s+1})\rightarrow K^m(Y_s).
\]
The exact couple giving rise to the spectral sequence for $B\Gamma$ is given by abelian groups $K^m(Y_s\times S^1)$ and $K^m(Y_s\times S^1,Y_{s-1}\times S^1)$.
We can write the groups in this exact couple as a $A_1^{s,m-s}\oplus A_1^{s,m-s-1}$ and $E_1^{s,m-s}\oplus E_1^{s,m-s-1}$.
The maps between these groups are sums of the maps between the exact couple for $B\Gamma'$.
Our inductive hypothesis therefore implies that the differentials vanish after the first page of the spectral sequence.
\end{proof}

\begin{remark}
When $c>0$, we do not need to know Theorem \ref{thm: K(BGamma  )}(\ref{part: diffs vanish}) in order to prove \ref{thm: K(BGamma  )}(\ref{part: convergence}) provided we have the $c=0$ case.
\end{remark}

Suppose Theorem \ref{thm: K(BGamma  )}(\ref{part: diffs vanish}) is true for groups of type $(a,b-1,0)$.
We will show this is true for $L$ a representation of type $(a,b,0)$.
First we introduce some notation.

\begin{notation}
\begin{itemize}
\item Given
\[
L=M_1\oplus\cdots\oplus M_a\oplus N_1\oplus\cdots \oplus N_b
\]
we make the following abreviations.
\begin{align*}
M_L&:=M_1\oplus \cdots \oplus M_a\\
N_L&:=N_1\oplus \cdots \oplus N_b
\end{align*}
\item For $d=1,\cdots, b$, define $L_d$ to be the representation of type $(a,b-1,0)$ where the $d$-th summand of $N_L$ is removed.
We have group homomorphisms
\[
\phi_d:L_d\rtimes\Z/p\rightarrow L\rtimes\Z/p\hspace{1cm} \psi_d:L\rtimes\Z/p\rightarrow L_d\rtimes\Z/p.
\]
Clearly, the composition $\psi_d\circ\phi_d$ is the identity.
We will henceforth denote $L_d\rtimes\Z/p$ by $\Gamma _d$.
It follows that there is the retraction of bundles below.
\begin{center}
\begin{tikzpicture}[scale=2]
\node (A) at (0,1) {$B\Gamma  _d$}; \node (B) at (1,1) {$B\Gamma $}; \node (C) at (2,1) {$B\Gamma _d$};
\node (D) at (0,0) {$B\Z/p$}; \node (E) at (1,0) {$B\Z/p$}; \node (F) at (2,0) {$B\Z/p$};
\path[->] (A) edge node[above]{$B\phi_{d}$} (B) (A) edge (D) (B) edge node[above]{$B\psi_{d}$} (C) (B) edge (E) (C) edge (F) (D) edge node[above]{$=$} (E) (E) edge node[above]{$=$} (F);
\end{tikzpicture}.
\end{center}
The maps $B\phi_{d}$ and $B\psi_{d}$ induce maps on cohomology, which we denote by $\phi_d^*$ and $\psi_d^*$.
The terms of the Atiyah-Hirzebruch-Serre spectral sequence for $B\Gamma _d$ will be denoted $E^{i,j}_{r,d}$.

\item By abuse of notation, we will define $\phi_d:\Z^{b-1}\rightarrow\Z^b$ and $\psi_d:\Z^b\rightarrow\Z^{b-1}$ similarly.

\item Let $\mathbf{r}=(r_1,\cdots,r_b)\in\Z^b$ be a $b$-tuple.
Define $A^{\mathbf{r}}:=\Lambda^{r_1} N_1^*\otimes\cdots\otimes\Lambda^{r_b}N_b^*$.
\item For a module $M$, let $\Lambda^{even}M$ be the sum of all $\Lambda^{2r}M$ and define $\Lambda^{odd}M$ similarly.

\item When one of the $r_i$ is neither $0$ nor $p$, then $\hat{H}^*\left(\Z/p;A^{\mathbf{r}}\right)=0$.
Indeed, Lemma \ref{lem: exterior is usually free} shows that $\Lambda^m\Z\left[\Z/p\right]$ is free when $1\le m\le p-1$ and the vanishing of the Tate cohomology follows from Proposition \ref{prop: simplifying Tate cohom}.
As we are not interested in all $\mathbf{r}\in\Z^b$, we define the following.
\[
\mc{R}_{b,m}:=\{(r_1,\cdots,r_b)\in\{0,p\}^b|r_1+\cdots+r_b\equiv m\text{ mod }2\}.
\]
For a subset $\mathbf{d}\subseteq\{1,\cdots,b\}$, define $\mc{R}_{b,m,\mathbf{d}}\subseteq\mc{R}_{b,m}$ to be those $b$-tuples such that $r_d=0$ if $d\in\mathbf{d}$.
\end{itemize}
\end{notation}

\begin{proof}[Proof of Theorem \ref{thm: K(BGamma  )} (1)]
We proceed by induction on $b$ with the case $b=0$ having been done in \cite[Lemma 3.3]{DavisLuckKTheory}.
Recall we are considering the Atiyah-Hirzebruch-Serre spectral sequence
\[
E_2^{i,j}=H^i\left(\Z/p;K^j\left(T^n_{\rho}\right)\right)\Rightarrow K^{i+j}(B\Gamma).
\]
We first check that, if $i>0$, then $d_2^{i,j}:E_2^{i,j}\rightarrow E_2^{i+2,j-1}$ is trivial.
Suppose that $j$ is even.
The term $E_2^{i,j}=H^i(\Z/p;\Lambda^{even}L^*)$ can be decomposed as a sum
\[
\bigoplus_{\mathbf{r}\in\mc{R}_{b,0}}H^i(\Z/p;\Lambda^{even}M_L^*\otimes A^{\mathbf{r}})\oplus\bigoplus_{\mathbf{r}\in\mc{R}_{b,1}}H^i(\Z/p;\Lambda^{odd}M_L^*\otimes A^{\mathbf{r}})
\]
and the term $E_2^{i+2,j-1}$ can be decomposed as a sum
\[
\bigoplus_{\mathbf{r}\in\mc{R}_{b,1}}H^{i+2}(\Z/p;\Lambda^{even}M_L^*\otimes A^{\mathbf{r}})\oplus\bigoplus_{\mathbf{r}\in\mc{R}_{b,0}}H^{i+2}(\Z/p;\Lambda^{odd}M_L^*\otimes A^{\mathbf{r}}).
\]

Moreover, we may identify the image of $\psi_d^*$ in $E_2^{i,j}$ with
\[
\bigoplus_{\mathbf{r}\in\mc{R}_{b,0,\{d\}}}H^i(\Z/p;\Lambda^{even}M_L^*\otimes A^{\mathbf{r}})\oplus\bigoplus_{\mathbf{r}\in\mc{R}_{b,1,\{d\}}}H^i(\Z/p;\Lambda^{odd}M_L^*\otimes A^{\mathbf{r}}).
\]

By the inductive hypothesis and the fact that $\phi_d^*\circ\psi_d^*:E_{r,d}^{i,j}\rightarrow E_{r,d}^{i,j}$ is the identity, these terms are in $E_{\infty}^{i,j}$.
It therefore remains to consider the effect of $d_2^{i,j}$ on the subgroup corresponding to $\mathbf{r}=(p,\cdots,p)$.
Thus we consider either a map
\[
H^i(\Z/p;\Lambda^{even}M_L^*)\rightarrow H^{i+2}(\Z/p;\Lambda^{odd}M_L^*)
\]
or
\[
H^i(\Z/p;\Lambda^{odd}M_L^*)\rightarrow H^{i+2}(\Z/p;\Lambda^{even}M_L^*)
\]
depending on the parity of $i$ and $b$.
In either case, Proposition \ref{prop: Davis-Luck group cohom} implies that these maps are $0$.
This completes the proof that $d_2^{i,j}=0$ when $i>0$ and when $j$ is even.
The case that $j$ is odd follows identically.

Now, suppose that $i=0$.
It follows from \cite[Lemma 1.10 (i)]{DavisLuckKTheory} that $H^2\left(\Z/p;\Lambda^{odd}M_L^*\otimes A^{(p,\cdots,p)}\right)=0$ so we just need to show that the restriction
\[
H^0\left(\Z/p;\Lambda^{odd}M_L^*\otimes A^{(p,\cdots,p)}\right)\rightarrow H^2\left(\Z/p;\Lambda^{even}M_L^*\otimes A^{(p,\cdots,p)}\right)
\]
of the differential must be trivial.
But by Proposition \ref{prop: norms and E_infty} and the fact that
\[
\hat{H}^0\left(\Z/p;\Lambda^{odd}M_L^*\otimes A^{(p,\cdots,p)}\right)=0,
\]
the left hand side must be in $E^{0,j}_\infty$.

The cases $r>2$ follows from a similar analysis.
\end{proof}

\subsection{Convergence}
In this section, we prove the second part of Theorem \ref{thm: K(BGamma  )} in the case $\Gamma$ is of type $(a,b,0)$.
The general case follows from this computation.
We induct by assuming that the second part is true for groups of type $(a,b',0)$ where $b'<b$.

Define the filtration for $K^m(B\Gamma )$ given by the spectral sequence by
\[
\cdots\subseteq F^2\subseteq F^1\subseteq F^0=K^m(B\Gamma ).
\]
Since $F^0/F^1=K^m\left(T_{\rho}^n\right)^{\Z/p}$ is a free abelian group, it suffices to compute $F^1$.

We first make a simplification.
Suppose that $N_L$ has rank $bp$ as an abelian group.
Let $\Gamma '$ be the group $\Gamma ':=(M_L\times\Z[\Z/p]^b)\rtimes\Z/p$.
Let $\cdots\subseteq G^2\subseteq G^1\subseteq G^0=K^m(B\Gamma ')$ denote the filtration on $K^m(B\Gamma ')$.

\begin{lemma}\label{lem: reduce to regular rep}
There is an isomorphism $F^1\cong G^1$.
\end{lemma}
\begin{proof}
Let $T^{bp}_{\Z[\Z/p]}$ denote the torus $T^{bp}$ with a $\Z/p$-action corresponding to the action of $\Z/p$ on the lattice $\Z[\Z/p]^b$.
Similarly, let $T^{bp}_{N_L}$ denote the torus with $\Z/p$-action corresponding to the lattice $N_L$.
Since $(N_L)_{(p)}\cong \Z[\Z/p]_{(p)}$, there is some $N$ prime to $p$ such that $T^{bp}_{N_L}$ is an $N$-sheeted regular cover of $T^{bp}_{\Z[\Z/p]}$.
From this, we obtain an $N$-sheeted regular cover $B\Gamma \rightarrow B\Gamma '$.
This is a map of bundles over $B\Z/p$ so this induces a map on the spectral sequences.
The induced maps $H^s\left(\Z/p;K^{m-s}\left(T^{a(p-1)}\times T^{bp}_{\Z[\Z/p]}\right)\right)\rightarrow H^s\left(\Z/p;K^{m-s}\left(T^{a(p-1)}\times T^{bp}_{N_L}\right)\right)$ are isomorphisms for $s\ge1$ since $N$ is prime to $p$.
Induction on $s$ via the diagram
\begin{center}
\begin{tikzpicture}
\node (A) at (0,1) {$G^s/G^{s+1}$}; \node (B) at (3,1) {$G^1/G^{s+1}$}; \node (C) at (6,1) {$G^1/G^s$};
\node (D) at (0,0) {$F^s/F^{s+1}$}; \node (E) at (3,0) {$F^1/F^{s+1}$}; \node (F) at (6,0) {$F^1/F^s$};
\path[->] (A) edge (B) (B) edge (C) (A) edge (D) (B) edge (E) (C) edge (F) (D) edge (E) (E) edge (F);
\end{tikzpicture}
\end{center}
shows that $\varprojlim F^1/F^s\cong \varprojlim G^1/G^s$.
It follows from Theorem \ref{thm: boardman convergence} that $F^1\cong G^1$.
\end{proof}

For the remainder of the section, we will assume
\[
\Gamma =(M_L\times\Z[\Z/p]^b)\rtimes\Z/p.
\]
In the computation of $F^1$, induction on $b$ will address the terms $H^i(\Z/p;\Lambda M_L^*\otimes A^{\mathbf{r}})$ when $\mathbf{r}\neq(p,\cdots,p)$.
In order to deal with the terms with $\mathbf{r}=(p,\cdots,p)$, we will need to consider a sphere bundle quotient of $B\Gamma $.
Recall we have assumed that $N_L\cong\Z[\Z/p]^b$ so $T^{bp}\cong\R[\Z/p]^b/N_L$.
Let $x_0\in T^{bp}$ denote the image of $0\in\R[\Z/p]^b$ and let $D$ denote a $\Z/p$-invariant disk neighborhood of $x_0$.
Then the quotient $T^{bp}/\left(T^{bp}\setminus D\right)$ is the representation sphere of the regular representation $\R[\Z/p]^b$.
Using this, we construct a map
\[
B\Gamma =T^{bp}\times_{\Z/p}\left(T^{a(p-1)}\times E\Z/p\right)\rightarrow S^{bp}\times_{\Z/p}\left(T^{a(p-1)}\times E\Z/p\right)
\]
of bundles over $B\Z/p$.
Let $E$ denote the sphere bundle $S^{bp}\times_{\Z/p}\left(T^{a(p-1)}\times E\Z/p\right)$.

\begin{lemma}\label{lem: Thom isomorphism}
There is an isomorphism
\[
K^m(E)\cong K^m(B(M_L\rtimes\Z/p))\oplus \tilde{K}^{m+b}(B(M_L\rtimes\Z/p)).
\]
Moreover the spectral sequence
\[
E_2^{i,j}=H^i(\Z/p;K^j(T^{a(p-1)}\times S^{bp}))\Rightarrow K^{i+j}(E)
\]
has trivial differentials $d_r^{i,j}$ for $r\ge2$.
\end{lemma}
\begin{proof}
There is a section $B(M_L\rtimes\Z/p)\rightarrow E$ such that $Th_0:=E/B(M_L\rtimes\Z/p)$ is the Thom space of the real vector bundle $\R[\Z/p]^b\times_{M_L\rtimes \Z/p}(EM_L\times E\Z/p)=\R[\Z/p]^b\times_{\Z/p}(T^{a(p-1)}\times E\Z/p)$.
This gives us an exact sequence
\[
\cdots\rightarrow\tilde{K}^0(Th_0)\rightarrow K^0(E)\rightarrow K^0(B(M_L\rtimes\Z/p))\rightarrow\tilde{K}^1(Th_0)\rightarrow K^1(E)\rightarrow\cdots
\]
in which the maps $K^m(E)\rightarrow K^m(B(M_L\rtimes\Z/p))$ are split surjective.
We obtain
\[
K^m(E)\cong K^m(B(M_L\rtimes\Z/p))\oplus\tilde{K}^m(Th_0).
\]

Note that $\R\otimes N_L\cong \R[\Z/p]^b$ so $\R\otimes N_L\cong \R^b\oplus V$ where $\R^b$ has trivial $\Z/p$ action and $V$ is a real $\Z/p$-representation obtained by forgetting the complex structure of a complex $\Z/p$-representation.
Writing $Th_1$ as the Thom space of $V\times_{\Z/p}(T^{a(p-1)}\times E\Z/p)$, we have $Th_0=\Sigma^b Th_1$.
We need to show that $\tilde{K}^m(B(M_L\rtimes\Z/p))\cong\tilde{K}^m(Th_1)$.

Define $E^n:=V\times_{\Z/p}\left(T^{a(p-1)}\times E\Z/p^{(n)}\right)$ where $E\Z/p^{(n)}$ denotes the $n$-skeleton of $E\Z/p$ and let $Th^n_1$ denote the Thom space of $E^n$ considered as a vector bundle over $T^{a(p-1)}\times_{\Z/p}E\Z/p^{(n)}$.
The Thom isomorphism for $K$-theory implies
\[
K^m(Th_1^n)\cong K^m(T^{a(p-1)}\times_{\Z/p}E\Z/p^{(n)}).
\]
This induces an isomorphism on inverse systems indexed by $n$.
In particular, $K^m(Th^n_1)$ is Mittag-Leffler.
Thus,
\[
\tilde{K}^m(Th_1)\cong\varprojlim_{n}\tilde{K}^m(Th^n_1)\cong \tilde{K}^m(B(M_L\rtimes\Z/p)).
\]

The proof of the second part is similar to the proof of the first part of Theorem \ref{thm: K(BGamma  )}.
\end{proof}

\begin{proof}[Proof of Theorem \ref{thm: K(BGamma  )} (2) Case $b=1$]
Denote the filtration on $K^m(E)$ coming from the fibration $E\rightarrow B\Z/p$ by
\[
\cdots\subseteq G^2\subseteq G^1\subseteq G^0=K^m(E).
\]
It follows from Lemma \ref{lem: Thom isomorphism} and the $b=0$ case of Theorem \ref{thm: K(BGamma  )} that $G^1\cong\hat{\Z}_p^{(p-1)p^a}$.

Suppose $i\ge1$ and $j$ is odd.
The $E_2^{i,j}$ term for the Atiyah-Hirzebruch-Serre spectral sequence for $K$-theory of the fibration $B\Gamma\rightarrow B\Z/p$ is
\[
E_2^{i,j}\cong H^i(\Z/p;\Lambda^{even}M_L^*\otimes\Lambda^p\Z[\Z/p])\oplus H^i(\Z/p;\Lambda^{odd}M_L^*\otimes \Lambda^0\Z[\Z/p]).
\]
On the other hand, the corresponding term for the fibration $E\rightarrow B\Z/p$ is
\[
E_2^{i,j}\cong H^i(\Z/p;\Lambda^{even}M_L^*\otimes H^p(S^p))\oplus H^i(\Z/p;\Lambda^{odd}M_L^*\otimes H^0(S^p)).
\]
The decompositions of the coefficients above follow from the K\"unneth formula for $K$-theory \cite{AtiyahKunneth}.
Naturality of the K\"unneth formula shows that these decompositions respect the $\Z/p$-module structures and that $B\Gamma \rightarrow E$ induces isomorphisms on $E_2^{i,j}$.
A similar argument shows that this is an isomorphism when $j$ is even as well.
Therefore, Theorem \ref{thm: boardman convergence} implies that $G^1\cong F^1$.
\end{proof}

\subsubsection{The $b>1$ case}
Assume now that $b>1$ and that Theorem \ref{thm: K(BGamma  )} is true for groups of type $(a,b',0)$ where $b'<b$.
We will need to make some more observations.
\begin{lemma}\label{lem: splitting of filtered ab groups}
Suppose we have filtered abelian groups
\begin{align*}
\cdots F^s\subseteq F^{s-1}\subseteq\cdots\subseteq& F^1\\
\cdots G^s\subseteq G^{s-1}\subseteq\cdots\subseteq& G^1
\end{align*}
with a filtration preserving split injection $G^1\rightarrow F^1$ whose splitting also preserves the filtration.
Then,
\[
\cdots F^s/G^s\subseteq F^{s-1}/G^{s-1}\subseteq\cdots\subseteq F^1/G^1
\]
is a filtered abelian group with slices
\[
(F^s/G^s)/(F^{s+1}/G^{s+1})\cong(F^s/F^{s+1})/(G^s/G^{s+1}).
\]
\end{lemma}
\begin{proof}
The splitting implies $G^1\cap F^s=G^s$, where we take intersections in $F^1$ and identify the $G^s$ with their images in $F^1$.
The result follows by considering the filtration
\[
\cdots F^s/(G^1\cap F^s)\subseteq F^{s-1}/(G^1\cap F^{s-1})\subseteq\cdots \subseteq F^1/G^1.
\]
\end{proof}

Denote the filtration on $K^m(B\Gamma _d)$ corresponding to the spectral sequence for $B\Gamma _d$ by
\[
\cdots\subseteq F^2_d\subseteq F^1_d\subseteq F^0_d=K^m(B\Gamma _d).
\]
The maps $F^s_d\rightarrow F^s$ are split injections.
In particular, $F^s_1\rightarrow F^s\rightarrow F^s/F_1$ is a split short exact sequence.
Now consider the square
\[
\begin{tikzpicture}[scale = 1.5]
\node (A) at (0,1) {$F^s_2$};\node (B) at (2,1) {$F^s$};
\node (C) at (0,0) {$F^s_2/(F^s_2\cap F^s_1)$};\node (D) at (2,0) {$F^s/F^s_1$};
\path[->] (A) edge (B) (A) edge (C) (B) edge (D) (C) edge (D);
\end{tikzpicture}.
\]
The right vertical map and the top map split which gives a splitting for the bottom map.
Therefore, there is a split short exact sequence
\[
0\rightarrow F^s_2/(F^s_2\cap F^s_1)\rightarrow F^s/F^s_1\rightarrow F^s/(F^s_1+F^s_2)\rightarrow0
\]
where $F^s_1+F^s_2$ denotes the subgroup of $F^s$ generated by $F^s_1$ and $F^s_2$.
Continuing this way, we obtain split short exact sequences
\[
0\rightarrow F^s_d/(F^s_d\cap(F^s_1+\cdots+F^s_{d-1})\rightarrow F^s/(F^s_1+\cdots+F^s_{d-1})\rightarrow F^s/(F^s_1+\cdots+F^s_d)\rightarrow 0.
\]
This shows that $F^s\rightarrow F^s/(F^s_1+\cdots +F^s_b)$ is split.
The group $F^s/(F^s_1+\cdots +F^s_b)$ is a filtration for the subgroup of $K^m(B\Gamma)$ with slices $H^i\left(\Z/p;(\Lambda N_L)\otimes A^{\mathbf{r}}\right)$ where at least one of $r_1,r_2,\cdots,r_d$ is $0$.

\begin{lemma}\label{lem: dim of G_1+...+G_k}
There is an isomorphism
\[
F^1_1+\cdots+F^1_b\cong\hat{\Z}^{(p-1)(p^a)\nu_b}_p
\]
where
\[
\nu_b:=\binom{b}{1}2^{b-2}-\binom{b}{2}2^{b-3}+\cdots+(-1)^{(b-1)+1}\binom{b}{b-1}2^0+\kappa_{b,m}.
\]
Here, $\kappa_{b,m}=(-1)^{b+1}$ when $m$ is even and $\kappa_{b,m}=0$ when $m$ is odd.
\end{lemma}
\begin{proof}
Let us abbreviate $F^1_{\alpha_1,\cdots,\alpha_d}:=F^1_{\alpha_1}\cap\cdots\cap F^1_{\alpha_d}$ where $\alpha_1,\cdots,\alpha_d$ are distinct integers in $\{1,\cdots,b\}$ and where the intersection occurs in $F^1\subseteq K^m(B\Gamma )$.
For $d<b$, each $F^1_{\alpha_1,\cdots,\alpha_d}$ is the image of the $p$-adic part of $K^m(B((M_L\oplus\Z[\Z/p]^{b-d})\rtimes\Z/p))$ under an appropriate retraction $B\Gamma \rightarrow B((M_L\oplus\Z[\Z/p]^{b-d})\rtimes\Z/p)$.
For $d=b$, the group $F^1_{1,\cdots,b}$ is the image of the $p$-adic part of $K^m(B(M_L\rtimes\Z/p))$ under the projection $B\Gamma \rightarrow B(M_L\rtimes\Z/p)$.

One checks that there is the following resolution.
\[
0\rightarrow F^1_{1,\cdots,b}\rightarrow\bigoplus_{\alpha=1}^b F^1_{1,\cdots,\hat{\alpha},\cdots,b}\rightarrow\cdots\rightarrow\bigoplus_{\alpha<\beta}F^1_{\alpha,\beta}\rightarrow\bigoplus_{\alpha=1}^bF^1_{\alpha}\rightarrow\sum_{\alpha=1}^b F^1_\alpha\rightarrow0
\]
Note that the $F^1_{\alpha_1,\cdots,\alpha_d}$ are finitely generated free $\hat{\Z}_p$-modules.
Moreover, any abelian group homomorphism of such modules is a $\hat{\Z}_p$-module homomorphism.\footnote{One examines the bijections $\Hom_{\Z}(\hat{\Z}_p,\hat{\Z}_p)=\varprojlim_{n}\Hom_{\Z}(\hat{\Z}_p,\Z/p^n)=\varprojlim_{n}\Hom{\Z}(\Z/p^n,\Z/p^n)=\varprojlim_{n}\Hom_{\Z}(\Z,\Z/p^n)=\Hom_{\Z}(\Z,\hat{\Z}_p)=\Hom_{\hat{\Z}_p}(\hat{\Z}_p,\hat{\Z}_p)$ where the second bijection follows from the fact that $\hat{\Z}_p$ has a unique subgroup of index $p^n$.}
As an abelian group, $\sum_{\alpha=1}^b F^1_\alpha$ is torsion-free; the retractions $\sum_{\alpha=1}^b F^1_\alpha\rightarrow F^1_\alpha$ give an injective homomorphism $\sum_{\alpha=1}^b F^1_\alpha\rightarrow\bigoplus_{\alpha=1}^b F^1_\alpha$ whose target is torsion-free.
Since $\hat{\Z}_p$ is a principal ideal domain, we see that $\sum_{\alpha=1}^bF^1_\alpha$ is a finitely generated free $\hat{\Z}_p$-module.

The result follows from tensoring with $\hat{\Q}_p$ and counting dimensions.
For $d=1,\cdots,b-1$, the induction hypothesis implies $F^1_{\alpha_1,\cdots,\alpha_d}\cong\hat{\Z}^{(p-1)p^a2^{d-1}}_p$.
For $d=b$, we have $F^1_{1,\cdots,b}\cong\hat{\Z}_p^{(p-1)p^a}\subseteq\tilde{K}^m(B(M_L\rtimes\Z/p))$, which accounts for the term $\kappa_{b,m}$.
\end{proof}

\begin{lemma}\label{lem: combo}
If $b\not\equiv m$ modulo $2$, $\nu_b=2^{b-1}$.
If $b\equiv m$ modulo $2$, $\nu_b=2^{b-1}-1$.
\end{lemma}
\begin{proof}
This follows from expanding $1=(2-1)^b$.
\end{proof}

\begin{proof}[Proof of Theorem \ref{thm: K(BGamma  )} (2) Case $b>1$]
Suppose that $b\not\equiv m$ modulo $2$.
Then $F^1/\left(\sum_{\alpha=1}^bF^1_\alpha\right)$ is filtered with slices
\[
\left(F^s/\left(\sum_{\alpha=1}^bF^s_\alpha\right)\right)/\left(F^{s+1}/\left(\sum_{\alpha=1}^bF^{s+1}_\alpha\right)\right)\cong\begin{cases}H^s(\Z/p;\Lambda^{even}M_L)&s\text{ odd}\\
H^s(\Z/p;\Lambda^{odd} M_L)&s\text{ even} \end{cases}.
\]
In either case, we see that these slices are trivial.
Since $\cap_{s=0}^\infty F^s=0$, we see that $F^1\cong\sum_{\alpha=1}^bF^1_\alpha$.

Now, suppose that $b\equiv m$ modulo $2$.
Then we have slices
\[
\left(F^s/\left(\sum_{\alpha=1}^bF^s_\alpha\right)\right)/\left(F^{s+1}/\left(\sum_{\alpha=1}^bF^{s+1}_\alpha\right)\right)\cong\begin{cases}H^s(\Z/p;\Lambda^{even}M_L\otimes A^{(p,\cdots,p)})&s\text{ even}\\
H^s(\Z/p;\Lambda^{odd} M_L\otimes A^{(p,\cdots,p)})&s\text{ odd} \end{cases}.
\]
Recall that we have assumed the submodule $N_L$ is isomorphic to $\Z[\Z/p]^b$ so there is the sphere bundle quotient $E$ discussed in Lemma \ref{lem: Thom isomorphism}.
As before, we let $\{G^s\}$ denote the filtration on $K^m(E)$.
Note that there is a split injection of filtered abelian groups $F^1_{1,\cdots, b}\rightarrow G^1$.
One checks that $G^1/F^1_{1,\cdots,b}\cong\hat{\Z}_p^{(p-1)p^a}$ and that the map of filtered abelian groups $G^1\rightarrow F^1\rightarrow F^1/(\sum_{\alpha=1}^bF^1_\alpha)$ factors through $G^1/F^1_{1,\cdots,b}$.
The map $G^1/F^1_{1,\cdots,b}\rightarrow F^1/(\sum_{\alpha=1}^bF^1_\alpha)$ induces an isomorphism on slices.
Using that $G^1\cong\varprojlim_{s}G^1/G^s$ and that $G^s\cong F^s_{1,\cdots,b}\oplus G^s/F^s_{1,\cdots,b}$ we get
\[
G^1/F^1_{1,\cdots,p}\cong\varprojlim_{s}\left(G^1/F^1_{1,\cdots,p}\right)/\left(G^s/F^s_{1,\cdots,p}\right)
\]
The isomorphism
\[
F^1/\left(\sum_{\alpha=1}^bF^1_\alpha\right)\cong \varprojlim_{s}\left(F^1/\left(\sum_{\alpha=1}^bF^1_\alpha\right)\right)/\left(F^{s}/\left(\sum_{\alpha=1}^bF^{s}_\alpha\right)\right).
\]
follows similarly.
Therefore, we obtain an isomorphism
\[
\hat{\Z}_p^{(p-1)p^a}\cong G^1/F^1_{1,\cdots,p}\cong F^1/\left(\sum_{\alpha}^bF^1_\alpha\right).
\]
Using Lemma \ref{lem: dim of G_1+...+G_k} and Lemma \ref{lem: combo}, we obtain $F^1\cong\hat{\Z}_p^{(p-1)p^a2^{b-1}}$ as desired.
\end{proof}

\subsection{Corollaries of the $K$-theory Computation}

We record some consequences of Theorem \ref{thm: K(BGamma  )} that we will need for computing the $L$-groups and the structure sets.
These results are proven for groups of type $(a,0,0)$ in \cite{DavisLuckKTheory} so we will assume that either $b\neq0$ or $c\neq 0$ in this section.
We need to import the following results, which can be found in \cite{DavisLuckKTheory}.

\begin{lemma}\label{lem: Ext of rep ring}
For a finite group $G$, there is an isomorphism
\[
\op{Ext}^i_{R_{\C}(G)}(M,R_{\C}(G))\cong\op{Ext}^i_{\Z}(M,\Z)
\]
for $i\ge0$.
\end{lemma}

\begin{theorem}[Universal Coefficients Theorem]\label{thm: UCT K-theory}
For any CW-complex and all $m\in\Z$, there is an short exact sequence
\[
0\rightarrow\op{Ext}^1_{\Z}(K_{m-1}(X),\Z)\rightarrow K^m(X)\rightarrow\Hom_{\Z}(K_m(X),\Z)\rightarrow0.
\]
Furthermore, when $X$ is finite, there is a an exact sequence
\[
0\rightarrow\op{Ext}^1_{\Z}(K^{m+1}(X),\Z)\rightarrow K_m(X)\rightarrow\Hom_{\Z}(K^m(X),\Z)\rightarrow0.
\]
These sequences are natural in $X$.
\end{theorem}

\begin{theorem}[Equivariant Universal Coefficients Theorem]\label{thm: UCT equiv K-theory}
Suppose $H$ is a finite group and that $X$ is an $H$-CW-complex.
For $m\in\Z$, there is a short exact sequence of $R_{\C}(H)$-modules
\[
0\rightarrow\op{Ext}^1_{R_{\C}(H)}(K_{m-1}^H(X),R_{\C}(H))\rightarrow K_H^m(X)\rightarrow \Hom_{R_{\C}(H)}(K_m^H(X),R_{\C}(H))\rightarrow 0.
\]
Furthermore, when $X$ is finite, there is an exact sequence
\[
0\rightarrow\op{Ext}^1_{R_{\C}(H)}(K_H^{m+1}(X),R_{\C}(H))\rightarrow K_m^H(X)\rightarrow \Hom_{R_{\C}(H)}(K^m_H(X),R_{\C}(H))\rightarrow0.
\]
These sequences are natural in $X$.
\end{theorem}

\begin{cor}\label{cor: K-homology differentials vanish}
The differentials in the Atiyah-Hirzebruch-Serre spectral sequence
\[
E^2_{i,j}=H_i\left(\Z/p;K_j\left(T^n_{\rho}\right)\right)\Rightarrow K_{i+j}(B\Gamma)
\]
vanish.
\end{cor}
\begin{proof}
The proof of this result is similar to the proof of Theorem \ref{thm: K(BGamma  )} so we only sketch it.
First, note that we may reduce to the case that $\Gamma$ is type $(a,b,0)$ as in Lemma \ref{lem: reduction of diffs vanishing to (a,b,0)}.

As a $\Z[\Z/p]$-module, $K_j\left(T^n_{\rho}\right)$ is isomorphic to the dual $K^j(T^n_{\rho})^*$.
Since dualization commutes with taking direct sums and dualization sends modules of type $(a,b,c)$ to modules of type $(a,b,c)$, the induction argument in the proof of Theorem \ref{thm: K(BGamma  )} proves that the differentials $d^2_{i,j}$ vanish when $i>2$.

Now, we check that the differentials 
\[
d^2_{2,j}:H_2\left(\Z/p;K_m\left(T_{\rho}^n\right)\right)\rightarrow H_0\left(\Z/p;K_m\left(T_{\rho}^n\right)\right)
\]
mapping to the left column vanish.
By the induction hypothesis, it suffices to show that the restriction of the differential summands of the form
\[
H_2\left(\Z/p;\left(\Lambda^{odd}M_L^*\otimes A^{(p,\cdots,p)}\right)^*\right)\rightarrow H_0\left(\Z/p;\left(\Lambda^{even}M_L^*\otimes A^{(p,\cdots,p)}\right)^*\right)
\]
and
\[
H_2\left(\Z/p;\left(\Lambda^{even}M_L^*\otimes A^{(p,\cdots,p)}\right)^*\right)\rightarrow H_0\left(\Z/p;\left(\Lambda^{odd}M_L^*\otimes A^{(p,\cdots,p)}\right)^*\right)
\]
vanish.
Since
\[
H_2\left(\Z/p;\left(\Lambda^{even}M_L^*\otimes A^{(p,\cdots,p)}\right)^*\right)\cong H^1\left(\Z/p;\Lambda^{even}M_L^*\otimes A^{(p,\cdots,p)}\right)\cong0,
\]
we only need to check the differentials vanish in the first case.
The left column consists of terms $K_j\left(T_{\rho}^n\right)_{\Z/p}$.
In order to show the differentials vanish, it suffices to show that the transgression $K_j\left(T_{\rho}^n\right)_{\Z/p}\rightarrow K_j(B\Gamma)$ is injective.
The norm map $K_j\left(T_{\rho}^n\right)_{\Z/p}\rightarrow K_j\left(T_{\rho}^n\right)^{\Z/p}$ factors through the transgression.
Since
\[
\hat{H}^{-1}\left(\Z/p;\left(\Lambda^{even}M_L^*\otimes A^{(p,\cdots,p)}\right)^*\right)\cong\hat{H}^1\left(\Z/p;\Lambda^{even}M_L^*\otimes A^{(p,\cdots,p)}\right)\cong0
\]
the norm map is injective on the summand $H_0\left(\Z/p;\left(\Lambda^{even}M_L^*\otimes A^{(p,\cdots,p)}\right)^*\right)$.
Hence, this term is in $E^{\infty}_{0,j}$.

The proof that $d^r_{i,j}$ vanishes for $r>2$ is similar.
\end{proof}

\begin{cor}\label{cor: K-homology of BGamma  }
There is an isomorphism
\[
K_m(B\Gamma )\cong \Hom_{\Z}(K^m(T_\rho^n)^{\Z/p},\Z)\oplus\left(\Z/p^\infty\right)^{(p-1)p^a2^{b+c-1}}
\]
where $\Hom_{\Z}\left(K^m\left(T_{\rho}^n\right)^{\Z/p},\Z\right)$ is the image of the map induced by the inclusion of the fiber $T_{\rho}^n\rightarrow B\Gamma $.
\end{cor}
\begin{proof}
Define $B^s:=T_\rho^n\times_{\Z/p}E\Z/p^{(s)}$.
We have the following direct system of short exact sequences.
\[
\begin{tikzpicture}[scale=1.75]
\node (A) at (0,3) {$\vdots$};\node (B) at (2,3) {$\vdots$};\node (C) at (4,3) {$\vdots$};
\node (D) at (0,2) {$\op{Ext}^1_{\Z}(K^{m+1}(B^s),\Z)$};\node (E) at (2,2) {$K_m(B^s)$};\node (F) at (4,2) {$\Hom_{\Z}(K^m(B^s),\Z)$};
\node (G) at (0,1) {$\op{Ext}^1_{\Z}(K^{m+1}(B^{s+1}),\Z)$};\node (H) at (2,1) {$K_m(B^{s+1})$};\node (I) at (4,1) {$\Hom_{\Z}(K^m(B^{s+1}),\Z)$};
\node (J) at (0,0) {$\vdots$}; \node (K) at (2,0) {$\vdots$};\node (L) at (4,0) {$\vdots$};
\path[->] (A) edge (D) (B) edge (E) (C) edge (F) (D) edge (E) (E) edge (F) (D) edge (G) (E) edge (H) (F) edge (I) (G) edge (H) (H) edge (I) (G) edge (J) (H) edge (K) (I) edge (L);
\end{tikzpicture}
\]
Taking the colimit, we obtain
\[
0\rightarrow\varinjlim\op{Ext}^1_{\Z}(K^{m+1}(B^s),\Z)\rightarrow K_m(B\Gamma )\rightarrow\varinjlim\Hom_{\Z}(K^m(B^s),\Z)\rightarrow0.
\]
By considering the Atiyah-Hirzebruch-Serre spectral sequence for the fibration $B^s\rightarrow B\Z/p^{(s)}$ and comparing it to that of the fibration $B\Gamma \rightarrow B\Z/p$, we see that $K^m(B^s)\cong K^m\left(T_{\rho}^n\right)^{\Z/p}\oplus A^s\oplus C^s$ where $A^s$ is some $p$-group and $C^s$ is a (possibly trivial) finitely generated free abelian group.
Indeed, the limit of the $A^s$ is exactly $F^1$ in the filtration of $K^m(B\Gamma )$.
Moreover, by considering morphisms of spectral sequences, $C^s$ is not in the image of $K^m(B^{s+1})$.
Therefore, the right hand term is isomorphic to $\Hom_{\Z}\left(K^m\left(T_{\rho}^n\right)^{\Z/p},\Z\right)\cong K^m\left(T_{\rho}^n\right)^{\Z/p}$.

The left hand term is isomorphic to $\varinjlim\op{Ext}^1_{\Z}(A^s,\Z)$ (we abuse notation here and let $A^s$ denote the $p$-group in $K^{m+1}(B^s)$).
We obtain isomorphisms
\[
\varinjlim\op{Ext}^1_{\Z}(A^s;\Z)\cong\varinjlim\widehat{A^s}\cong\widehat{\varprojlim A^s}\cong\widehat{\left(\hat{\Z}_p^{(p-1)p^a2^{b+c-1}}\right)}\cong\left(\Z/p^{\infty}\right)^{(p-1)p^a2^{b+c-1}}
\]
where $\widehat{A}$ denotes the Pontryagin dual of a locally compact abelian group $A$.
We refer to the proof of \cite[Theorem 4.1]{DavisLuckKTheory} and \cite{Vick} for details regarding Pontryagin duality.

It remains to check that the subgroup $\Hom_{\Z}\left(K^m\left(T_{\rho}^n\right)^{\Z/p},\Z\right)$ is the image of the map induced by $T_\rho^n\rightarrow B\Gamma $.
The inclusion induces the composition
\[
K^m(B^s)\rightarrow K^m\left(T_{\rho}^n\right)^{\Z/p}\inj K^m\left(T_{\rho}^n\right).
\]
The result follows by the commutativity of the diagram
\[
\begin{tikzpicture}[scale=2]
\node (A) at (0,1) {$K_m\left(T_{\rho}^n\right)$};\node (B) at (2,1) {$\Hom_{\Z}\left(K^m\left(T_{\rho}^n\right),\Z\right)$};
\node (C) at (0,0) {$K_m(B\Gamma )$};\node (D) at (2,0) {$\varinjlim\Hom_{\Z}(K^m(B^s),\Z)$};
\path[->] (A) edge (B) (A) edge (C) (C) edge (D) (B) edge (D);
\end{tikzpicture}.
\]
\end{proof}

In the future, we will write $K^m\left(T_{\rho}^n\right)^{\Z/p}$ rather than $\Hom_{\Z}\left(K^m\left(T_{\rho}^n\right)^{\Z/p},\Z\right)$.

\begin{cor}\label{cor: KO-homology of BGamma  }
After inverting $2$, $KO_m(B\Gamma )$ is the sum of a finitely generated free $\Z\left[\frac{1}{2}\right]$-module and a $p$-torsion group.
Moreover, the inclusion $T_{\rho}^n\rightarrow B\Gamma $ induces a surjection on the finitely generated free $\Z\left[\frac{1}{2}\right]$-module.
\end{cor}
\begin{proof}
Consider the following diagram.
\[
\begin{tikzpicture}[scale=1.5]
\node (A) at (0,1) {$KO_m\left(T_{\rho}^n\right)$};\node (B) at (2,1) {$K_m\left(T_{\rho}^n\right)$};\node (C) at (4,1) {$KO_m\left(T_{\rho}^n\right)$};
\node (D) at (0,0) {$KO_m(B\Gamma )$};\node (E) at (2,0) {$K_m(B\Gamma )$};\node (F) at (4,0) {$KO_m(B\Gamma )$};
\path[->] (A) edge node[above]{$i_*$} (B) (B) edge node[above]{$r_*$} (C) (A) edge (D) (B) edge (E) (C) edge (F) (D) edge node[above]{$i_*$} (E) (E) edge node[above]{$r_*$} (F);
\end{tikzpicture}
\]
The horizontal composites are multiplication by $2$.
Thus, after inverting $2$, $i_*$ is injective.
Applying Corollary \ref{cor: K-homology of BGamma  } proves the first part.

For the second part, let $x\in KO_m(B\Gamma )$ be an element in the finitely generated free $\Z\left[\frac{1}{2}\right]$-submodule of $KO_m(B\Gamma )$.
Then, $i_*x$ is in the image of the middle vertical map by Corollary \ref{cor: K-homology of BGamma }.
It pulls back to an element $y\in K_m\left(T_\rho^n\right)$.
But then $x$ is the image of $\frac{1}{2}r_*y$ under the outer vertical maps.
\end{proof}

\begin{cor}\label{cor: equiv K theory torsion free}
The groups $K^m_\Gamma (\underline{E}\Gamma )$, $KO^m_\Gamma (\underline{E}\Gamma )$ and $KO_m^\Gamma (\underline{E}\Gamma )$ are $p$-torsion free.
\end{cor}
\begin{proof}
First, we show that $K^m_\Gamma (\underline{E}\Gamma )$ is $p$-torsion free.
Let $\tilde{R}_{\C}(\Z/p)$ denote the reduced complex representation ring of the group $\Z/p$.
Proposition \ref{prop: LES from pushout diagram} gives us the top row of the following diagram.
\[
\begin{tikzpicture}[scale=1.5]
\node (A) at (0,1) {$\cdots$};\node (B) at (2,1) {$\bigoplus_{(P)\in\mc{P}}\tilde{R}_{\C}(\Z/p)^{2^{b+c-1}}$};\node (C) at (4,1) {$K^m(\underline{B}\Gamma )$};\node (D) at (6,1) {$K_\Gamma ^m(\underline{E}\Gamma )$};\node (E) at (8,1) {$\oplus_{(P)\in\mc{P}}\tilde{R}_{\C}(\Z/p)^{2^{b+c-1}}$};\node (F) at (10,1) {$\cdots$};
\node (G) at (0,0) {$\cdots$};\node (H) at (2,0) {$\bigoplus_{(P)\in\mc{P}}\hat{\Z}_p^{2^{b+c-1}}$};\node (I) at (4,0) {$K^m(\underline{B}\Gamma )$};\node (J) at (6,0) {$K^m(B\Gamma )$};\node (K) at (8,0) {$\bigoplus_{(P)\in\mc{P}}\hat{\Z}_p^{2^{b+c-1}}$};\node (L) at (10,0) {$\cdots$};
\path[->] (A) edge (B) (B) edge (C) (C) edge (D) (D) edge (E) (E) edge (F) (G) edge (H) (H) edge node[above]{$\phii$} (I) (I) edge (J) (J) edge (K) (K) edge (L) (B) edge (H) (C) edge node[left]{$=$} (I) (D) edge (J) (E) edge (K);
\end{tikzpicture}
\]
The bottom row comes from applying $K$-theory to the homotopy pushout diagram obtained from the quotient of diagram \ref{diagram: pushout for EGamma } by $\Gamma $.
The vertical map $\bigoplus_{(P)\in\mc{P}}\tilde{R}_{\C}(\Z/p)^{2^{b+c-1}}\rightarrow\bigoplus_{(P)\in\mc{P}}\hat{\Z}_p^{2^{b+c-1}}$ is $p$-adic completion.

Let $x\in K^m_\Gamma (\underline{E}\Gamma )$ be an element of order $p$.
Then, $x$ must pull back to an element in $K^m(\underline{B}\Gamma )$ which then pulls back to and element $\overline{x}\in\bigoplus_{(P)\in\mc{P}}\hat{\Z}_p^{2^{b+c-1}}$ (here we use that $K^m(B\Gamma )$ is torsion free).
Using the transfer, one can check that $K^m(\underline{B}\Gamma )$ is the sum of a finitely generated free abelian group with a finite $p$-group.
The image of $\phii$ must be contained in the $p$-group (it is the entire $p$-group as $K^m(B\Gamma )$ is torsion free).
Thus $\phii$ factors through $\oplus_{(P)\in\mc{P}}\hat{\Z}_p^{2^{b+c-1}}/p^N\hat{\Z}_p^{2^{b+c-1}}$ for some $N$.
But every element in this quotient can be represented by an element in the image of $\bigoplus_{(P)\in\mc{P}}\tilde{R}_{\C}(\Z/p)^{2^{b+c-1}}$.
Let $\tilde{x}\in\bigoplus_{(P)\in\mc{P}}\tilde{R}_{\C}(\Z/p)^{2^{b+c-1}}$ be an element lifting the projection of $\overline{x}$.
Then, we obtain that $\tilde{x}$ maps to $x\in K_\Gamma ^m(\underline{E}\Gamma )$ but exactness implies that $x=0$.
This shows that $K^m_\Gamma (\underline{E}\Gamma )$ has no $p$-torsion.

Lemma \ref{lem: Ext of rep ring} and Theorem \ref{thm: UCT equiv K-theory} imply that $K^\Gamma _m(\underline{E}\Gamma )$ has no $p$-torsion.
Since multiplication by $2$ in $KO^\Gamma _m(\underline{E}\Gamma )$ factors through $K^\Gamma _m(\underline{E}\Gamma )$, it follows that $KO^\Gamma _m(\underline{E}\Gamma )$ has no $p$-torsion.
\end{proof}

\section{$L$-Theory Computations}\label{section: L theory}
For geometric applications, one is typically interested in the groups $L^s_m(\Z[\Gamma])$ and $L^h_m(\Z[\Gamma])$.
The group $L^s_m(\Z[\Gamma])$ contains obstructions to obtaining simple homotopy equivalences through surgery and the group $L^h_m(\Z[\Gamma])$ contains obstructions to obtaining homotopy equivalences through surgery.
There is a map $L^s_m(\Z[\Gamma])\rightarrow L^h_m(\Z[\Gamma])$.
More generally, one can define the lower $L$-groups $L^{\langle j\rangle}_m(\Z[\Gamma])$ where $j=2,1,0,\cdots$ with the convention
\[
L^{\langle 2\rangle}_m(\Z[\Gamma])=L^s_m(\Z[\Gamma])\hspace{1cm}L^{\langle 1\rangle}_m(\Z[\Gamma])=L^h_m(\Z[\Gamma]).
\]
There are maps
\[
L^{\langle j\rangle}_m(\Z[\Gamma])\rightarrow L^{\langle j-1\rangle}_m(\Z[\Gamma])
\]
and we define $L^{\langle-\infty\rangle}_m(\Z[\Gamma]):=\varinjlim L^{\langle j\rangle}_m(\Z[\Gamma])$.
This theory is developed in \cite{RanickiLowerLTheory}.
The group $L^{\langle j\rangle}_m(\Z[\Gamma])$ is $\pi_m L^{\langle j\rangle}(\Z[\Gamma])$ for a spectrum $L^{\langle j\rangle}(\Z[\Gamma])$.
Also, the functors $\mathbf{L}^{\langle j\rangle}_{\Z}:Grpd\rightarrow Sp$ send a group $G$ (regarded as a groupoid) to $L^{\langle j\rangle}(\Z[G])$.

We begin by computing $L^{\langle-\infty\rangle}_m(\Z[\Gamma])$.
This is easier to work with as the Farrell-Jones Conjecture holds for $L^{\langle-\infty\rangle}$.
Using Rothenberg sequences \cite[Section 17]{RanickiLowerLTheory}, we then compute $L^{\langle j\rangle}_m(\Z[\Gamma])$ for all $j$.

One of the primary $L$-groups that appear in our computations are the groups $L^{\langle j\rangle}_m\left(\Z\left[N_\Gamma P\right]\right)$ and $L^{\langle j\rangle}_m\left(\Z\left[W_\Gamma P\right]\right)$ so we take some time to discuss these groups here.
Recall that $N_\Gamma P\cong\Z^{b+c}\times\Z/p$ and $W_\Gamma P\cong\Z^{b+c}$.
Shaneson splitting \cite[Theorem 17.2]{RanickiLowerLTheory} gives isomorphisms
\[
L^{\langle j\rangle}_m\left(\Z\left[N_\Gamma P\right]\right)\cong\bigoplus_{i=0}^{b+c}L^{\langle j-i\rangle}_{m-i}(\Z[\Z/p])^{\binom{b+c}{i}}\hspace{1cm}L^{\langle j\rangle}_m\left(\Z\left[W_\Gamma P\right]\right)\cong\bigoplus_{i=0}^{b+c}L^{\langle j-i\rangle}_{m-i}(\Z)^{\binom{b+c}{i}}.
\]
When $j=-\infty$, the Farrell-Jones Conjecture says that the first group can be repackaged as the homology group $H_m\left(T^b;L^{\langle-\infty\rangle}(\Z[\Z/p])\right)$.
The second group can always be rewritten as $H_m\left(T^{b+c};L(\Z)\right)$ since the maps $L^{\langle j\rangle}_m(\Z)\rightarrow L^{\langle j-1\rangle}_m(\Z)$ are isomorphisms for all $j$.
Since the map $N_\Gamma P\rightarrow W_\Gamma P$ splits in our case, we have an inclusion $L^{\langle j\rangle}_m\left(\Z\left[W_\Gamma P\right]\right)\rightarrow L^{\langle j\rangle}_m\left(\Z\left[N_\Gamma P\right]\right)$.
The quotient is
\[
L^{\langle j\rangle}_m\left(\Z\left[N_\Gamma P\right]\right)/L^{\langle j\rangle}_m\left(\Z\left[W_\Gamma P\right]\right)\cong\bigoplus_{i=0}^{b+c}\tilde{L}^{\langle j-i\rangle}_{m-i}(\Z[\Z/p])^{\binom{b+c}{i}}.
\]
So, these groups can be computed in terms of the $L$-groups of the group $\Z/p$.

Finally, we record the $L$-groups of $\Z/p$.
The following theorem can be found in \cite{BakEven} and \cite{BakOdd}.
\begin{theorem}\label{thm: L groups of Z/p}
There are isomorphisms
\begin{align*}
\tilde{L}^s_m(\Z[\Z/p])&\cong\begin{cases}\Z^{(p-1)/2}&m \text{ even}\\
0&m\text{ odd}
\end{cases}\\
\tilde{L}^h_m(\Z[\Z/p])&\cong\begin{cases}\Z^{(p-1)/2}\oplus H(\Z/p)&m \text{ even}\\
0&m\text{ odd}
\end{cases}
\end{align*}
where $H(\Z/p)$ is finitely generated abelian group of exponent $2$.
For $j\le0$, there is an isomorphism
\[
\tilde{L}^{\langle j\rangle}_m(\Z[\Z/p])\cong\begin{cases}\Z^{(p-1)/2}&m\text{ even}\\
0&m\text{ odd}
\end{cases}.
\]
\end{theorem}
\begin{remark}
Although $\tilde{L}^s_m(\Z[\Z/p])$ and $\tilde{L}^{\langle j\rangle}_m(\Z[\Z/p])$ are isomorphic for $j\le 0$, the map
\[
\tilde{L}^s_m(\Z[\Z/p])\rightarrow\tilde{L}^{\langle j\rangle}_m(\Z[\Z/p])
\]
is not an isomorphism.
\end{remark}

\subsection{The $L^{\langle-\infty\rangle}$ Computation}

\begin{theorem}\label{thm: L^-infty computation}
There is an isomorphism
\[
L_m^{\langle -\infty\rangle}(\Z[\Gamma])\cong\left(\bigoplus_{(P)\in\mc{P}}H_m\left(\left(T_\rho^n\right)^{P};\tilde{L}^{\langle -\infty\rangle}(\Z[\Z/p])\right)\right)\oplus H_m\left(T_{\rho}^n;L(\Z)\right)^{\Z/p}.
\]
\end{theorem}
\begin{proof}
We show that the isomorphism holds when $p$ is inverted and when $2$ is inverted.
This suffices since the groups are finitely generated.

Using the Proposition \ref{prop: transitivity} and Proposition \ref{prop: LES from pushout diagram}, we have the following long exact sequence.
\begin{equation}\label{les: L^-infty homology}
\cdots\rightarrow\mc{K}_m\xrightarrow{\phii_m}L^{\langle-\infty\rangle}_m(\Z[\Gamma ])\rightarrow H_m(\underline{B}\Gamma ;L(\Z))\rightarrow\mc{K}_{m-1}\rightarrow\cdots
\end{equation}
By the remark after Proposition \ref{prop: LES from pushout diagram}, we can make the identification 
\[
\mc{K}_m\cong\bigoplus_{(P)\in\mc{P}}H_m\left(\left(T_\rho^n\right)^P;\tilde{L}^{\langle-\infty\rangle}(\Z[\Z/p])\right).
\]
After inverting $p$, the sequence splits into short exact sequences
\[
0\rightarrow\mc{K}_m\rightarrow L_m^{\langle-\infty\rangle}(\Z[\Gamma ])\rightarrow H_m(\underline{B}\Gamma ;L(\Z))\rightarrow0.
\]
But when $p$ is inverted, the right hand term is isomorphic to $H_m\left(T_\rho^n;L(\Z)\right)^{\Z/p}$ by a transfer argument \cite[Proposition A.4]{DavisLuckKTheory}.
This is free so the sequence splits.
Note that $\mc{K}_m$ is a free abelian group since $\tilde{L}^{\langle-\infty\rangle}_m(\Z[\Z/p])$ is free abelian.

Now it remains to show that the groups are isomorphic after inverting $2$.
By \cite[Theorem 4.2]{DavisLuckTorusBundles}, equivariant $L$-theory homology and equivariant $KO$-homology agree after inverting $2$.
We obtain the resulting long exact sequence.
\[
\cdots\rightarrow\mc{K}_m\left[\frac{1}{2}\right]\xrightarrow{\phii_m} KO^\Gamma_m(\underline{E}\Gamma )\left[\frac{1}{2}\right]\rightarrow KO_m(\underline{B}\Gamma )\left[\frac{1}{2}\right]\rightarrow\mc{K}_{m-1}\left[\frac{1}{2}\right]\rightarrow\cdots
\]
We can write $KO_m(\underline{B}\Gamma )\left[\frac{1}{2}\right]\cong F\oplus A$ where $F$ is a free $\Z\left[\frac{1}{2}\right]$-module and $A$ is a $p$-torsion group.
The map $KO^\Gamma_m(\underline{E}\Gamma )\left[\frac{1}{2}\right]\rightarrow KO_m(\underline{B}\Gamma )\left[\frac{1}{2}\right]$ is invertible after inverting $p$ so there is a partial section defined on a $p$-power index subgroup of $KO_m(\underline{B}\Gamma )\left[\frac{1}{2}\right]$.
Since $KO_m^\Gamma (\underline{E}\Gamma )$ has no $p$-torsion, this subgroup must be a $p$-power index subgroup of $F$, hence isomorphic to $F$.
This partial splitting gives a subgroup $\mc{K}_m\left[\frac{1}{2}\right]\oplus F$, which is $p$-power index in $KO_m^\Gamma (\underline{E}\Gamma )\left[\frac{1}{2}\right]$.
Therefore, there is an isomorphism
\[
KO_m^\Gamma (\underline{E}\Gamma )\left[\frac{1}{2}\right]\cong\mc{K}_m\left[\frac{1}{2}\right]\oplus F.
\]
It remains to check that $F$ is isomorphic to $H_m(T_{\rho}^n;L(\Z))^{\Z/p}\left[\frac{1}{2}\right]$.
But this follows from the fact that $H_m\left(T_{\rho}^n;L(\Z)\right)^{\Z/p}\left[\frac{1}{2}\right]$ is a free $\Z\left[\frac{1}{2}\right]$-module isomorphic to a $p$-power index subgroup of $H_m(\underline{B}\Gamma ;L(\Z))\left[\frac{1}{2}\right]$.
\end{proof}

\subsection{Arbitrary Decorations}
For the groups studied in \cite{DavisLuckTorusBundles}, the assembly map $H^G_m\left(\underline{E}G;\mathbf{L}_{\Z}^{\langle j\rangle}\right)\rightarrow L^{\langle j\rangle}_m(\Z[G])$ is an isomorphism for all decorations $j$.
This is essentially because the normalizers of finite subgroups are isomorphic to $\Z/p$ and because the analogous result holds for $\Z/p$.
In our case, the normalizers are of the form $\Z^{b+c}\times\Z/p$ so the situation becomes more complicated.

In order to study $L$-theory with arbitrary decorations, we need to use Whitehead groups.
\begin{definition}\label{def: whitehead}
For a group $G$ and integer $m$, define the Whitehead group $\op{Wh}_m(G)$ to be
\[
\op{Wh}_m(G):=H_m^G\left(EG\rightarrow pt;\mathbf{K}_\Z\right).
\]
\end{definition}

The isomorphism in the statement below is \cite[3.29]{LuckRosenthal}.
That this isomorphism is induced by an inclusion of subgroups follows from the proof of \cite[Theorem 1.10]{LuckRosenthal}.
\begin{theorem}\label{thm: whitehead computation}
There is an isomorphism
\[
\op{Wh}_m(\Gamma )\cong\bigoplus_{(P)\in\mc{P}}\op{Wh}_m\left( N_\Gamma P\right)
\]
induced by the inclusions $N_{\Gamma}P\rightarrow\Gamma$.
\end{theorem}

Since $\op{Wh}_m(\Z^{b+c}\times\Z/p)\cong\bigoplus_{k=0}^{b+c}\op{Wh}_{m-k}(\Z/p)^{\binom{b+c}{k}}$ and $\op{Wh}_m(\Z/p)=0$ when $m\le-1$, $\op{Wh}_m(\Gamma)=0$ for $m\le-1$.

\Lj

\begin{proof}
The homology group $H_m\left((T_{\rho}^n)^P;\tilde{L}^{\langle-\infty\rangle}(\Z[\Z/p])\right)$ fits into the exact sequence
\[
\cdots\rightarrow H_m\left((T_{\rho}^n)^P;L(\Z)\right)\rightarrow H_m\left((T_{\rho}^n)^P;L^{\langle-\infty\rangle}(\Z[\Z/p])\right)\rightarrow H_m\left((T_{\rho}^n)^P;\tilde{L}^{\langle-\infty\rangle}(\Z[\Z/p])\right)\rightarrow\cdots
\]
where the map $L(\Z)\rightarrow L^{\langle-\infty\rangle}(\Z[\Z/p])$ splits.
By the Farrell-Jones conjecture (or Shaneson splitting), we may identify
\[
H_m\left((T_{\rho}^n)^P;L(\Z)\right)\cong L_m(\Z[W_\Gamma P])\hspace{1cm} H_m\left((T_{\rho}^n)^P;L^{\langle-\infty\rangle}(\Z[\Z/p])\right)\cong L_m^{\langle-\infty\rangle}(\Z[N_\Gamma P]).
\]
It follows from Theorem \ref{thm: L^-infty computation} and the resulting identification
\[
H_m\left(T_{\rho}^n)^P;\tilde{L}^{\langle-\infty\rangle}(\Z[\Z/p])\right)\cong L^{\langle-\infty\rangle}_m(\Z[N_\Gamma P])/L_m^{\langle-\infty\rangle}(\Z[W_\Gamma P])
\]
that the result is true when $2$ is inverted.
Now, we need to check that the result is true when $p$ is inverted.
Since $\op{Wh}_j(N_\Gamma P),\op{Wh}_j(\Gamma)=0$ for $j\le-1$, the Farrell-Jones Conjecture and Proposition \ref{prop: LES from pushout diagram} imply that
\[
\bigoplus_{(P)\in\mc{P}}L_m^{\langle j\rangle}(\Z[N_\Gamma P])/L_m^{\langle j\rangle}(\Z[W_\Gamma P])\rightarrow L_m^{\langle j\rangle}(\Z[\Gamma])\xrightarrow{\beta^{\langle j\rangle}_m}H_m(\underline{B}\Gamma;L(\Z))
\]
is a split short exact sequence when $p$ is inverted.

We claim the same is true for $j>-1$.
To verify this claim, we induct on $j$, using $j=-1$ as the base case.
Consider the following diagram
\[
\begin{tikzpicture}[scale =2.5]
\node (A) at (0,4) {$\vdots$};\node (B) at (2,4) {$\vdots$};\node (X) at (4,4) {$\vdots$};
\node (C) at (0,3) {$\bigoplus_{(P)\in\mc{P}}\hat{H}^{m+1}\left(\Z/2;\op{Wh}_j(N_\Gamma P)\right)$};\node (D) at (2,3) {$\hat{H}^{m+1}\left(\Z/2;\op{Wh}_j(\Gamma)\right)$};\node (E) at (4,3) {$0$};
\node (F) at (0,2) {$\bigoplus_{(P)\in\mc{P}}L^{\langle j+1\rangle}_m(\Z[N_\Gamma P])/L_m^{\langle j+1\rangle}(\Z[W_\Gamma P])$};\node (G) at (2,2) {$L^{\langle j+1\rangle}_m(\Z[\Gamma])$};\node (H) at (4,2) {$H_m(\underline{B}\Gamma;L(\Z))$};
\node (I) at (0,1) {$\bigoplus_{(P)\in\mc{P}}L^{\langle j\rangle}_m(\Z[N_\Gamma P])/L_m^{\langle j\rangle}(\Z[W_\Gamma P])$};\node (J) at (2,1) {$L^{\langle j\rangle}_m(\Z[\Gamma])$};\node (K) at (4,1) {$H_m(\underline{B}\Gamma;L(\Z))$};
\node(L) at (0,0) {$\vdots$};\node (M) at (2,0) {$\vdots$};\node (Y) at (4,0) {$\vdots$};
\path[->] (A) edge (C) (B) edge (D) (C) edge (D) (D) edge (E) (C) edge (F) (D) edge (G) (E) edge (H) (F) edge (G) (G) edge node[above]{$\beta^{\langle j+1\rangle}_m$} (H) (F) edge (I) (G) edge (J) (H) edge node[left]{$=$} (K) (I) edge (J) (J) edge node[above]{$\beta^{\langle j\rangle}_m$} (K) (I) edge (L)(J) edge (M) (X) edge (E) (K) edge (Y);
\end{tikzpicture}
\]

The left vertical column comes from the Rothenberg sequence and the observation that $W_\Gamma P$ is a free abelian group which splits from $N_\Gamma P$.
The middle column is a Rothenberg sequence.
The map $\beta^{\langle j+1\rangle}_m$ is defined so that the diagram commutes.
If we invert $p$, the bottom row is split exact by hypothesis.
Then, the middle row is exact.

Now, we must check that $\beta^{\langle j+1\rangle}_m$ is split when $p$ is inverted.
We have the commuting diagram
\[
\begin{tikzpicture}[scale =2]
\node (A) at (0,3) {$H^{\Gamma}_m\left(\underline{E}\Gamma;\mathbf{L}^{\langle j+1\rangle}_\Z\right)$};\node (B) at (2,3) {$H_m(\underline{B}\Gamma;L(\Z))$};
\node (C) at (0,2) {$H^{\Gamma}_m\left(\underline{E}\Gamma;\mathbf{L}^{\langle-\infty\rangle}_\Z\right)$};\node (D) at (2,2) {$H_m(\underline{B}\Gamma;L(\Z))$};
\node (E) at (0,1) {$L^{\langle-\infty\rangle}_m(\Z[\Gamma])$};\node (F) at (2,1) {$H_m(\underline{B}\Gamma;L(\Z))$};
\node (G) at (0,0) {$L^{\langle j+1\rangle}_m(\Z[\Gamma])$};\node (H) at (2,0) {$H_m(\underline{B}\Gamma;L(\Z))$};
\path[->] (A) edge node[above]{$\overline{\beta}^{\langle j+1\rangle}_m$} (B) (A) edge (C) (B) edge (D) (C) edge (D) (C) edge node[right]{$A^{\langle-\infty\rangle}_m$} (E) (D) edge (F) (E) edge (F) (G) edge (E) (H) edge (F) (G) edge node[above]{$\beta^{\langle j+1\rangle}_m$}(H) (A) edge [bend right=65] node[left]{$A^{\langle j+1\rangle}_m$} (G);
\end{tikzpicture}.
\]
where the maps $A^{\langle-\infty\rangle}_m$ and $A^{\langle j\rangle}_m$ are assembly maps.
The top map comes from Proposition \ref{prop: LES from pushout diagram} so it splits after $p$ is inverted.
The vertical maps on the right are equalities.
The map $L_m^{\langle-\infty\rangle}(\Z[\Gamma])\rightarrow H_m(\underline{B}\Gamma;L(\Z))$ is chosen so the middle square commutes which can be done since $A_m^{\langle-\infty\rangle}$ is an isomorphism.
The bottom square commutes from the definition of $\beta_m^{\langle j+1\rangle}$.
The composite $A^{\langle j+1\rangle}_m\circ\left(\overline{\beta}^{\langle j+1\rangle}_m\right)^{-1}$ gives us the desired splitting of $\beta^{\langle j+1\rangle}_m$.

Thus, when $p$ is inverted, we have split short exact sequences
\begin{equation}\label{eq: L split with deco}
\bigoplus_{(P)\in\mc{P}}L_m^{\langle j+1\rangle}(\Z[N_\Gamma P])/L_m^{\langle j+1\rangle}(\Z[W_\Gamma P])\rightarrow L_m^{\langle j+1\rangle}(\Z[\Gamma])\xrightarrow{\beta_m^{\langle j+1\rangle}}H_m(\underline{B}\Gamma;L(\Z)).
\end{equation}
As $H_m(\underline{B}\Gamma;L(\Z))\left[\frac{1}{p}\right]\cong\left(H_m(T_{\rho}^n;L(\Z))^{\Z/p}\right)\left[\frac{1}{p}\right]$, this finishes the proof of the theorem.
\end{proof}

\section{Computation of the Structure Set}\label{section: structure set}

\subsection{A Brief Review of Surgery}
Let $M$ be a simple Poincar\'e duality complex.
Its \emph{geometric simple structure set}, denoted $\mc{S}^{geo,s}(M)$, is defined to be the set of equivalence classes of simple homotopy equivalences $f:N\rightarrow M$.
Two simple homotopy equivalences $f:N\rightarrow M$ and $f':N'\rightarrow M$ are equivalent if there is a homeomorphism $g:N\rightarrow N'$ such that $f$ and $f'\circ g$ are homotopic.
The geometric simple structure set is contained in geometric surgery exact sequence
\[
\cdots\rightarrow\mc{N}(M\times I)\rightarrow L^s_{n+1}(\Z[\pi_1(M)])\rightarrow\mc{S}^{geo,s}(M)\rightarrow\mc{N}(N)\rightarrow L^s_n(\Z[\pi_1(M)]).
\]
This can be found in \cite{WallSCM}.
It is not clear from this description that there is an abelian group structure on $\mc{S}^{geo,s}(M)$.

In \cite{RanickiLTheory}, Ranicki defines an algebraic surgery exact sequence valid for any CW complex $X$.
\begin{align*}
\cdots\rightarrow H_{m+1}(X;L(\Z))\xrightarrow{A_{m+1}(X)} &L^s_{m+1}(\Z[\pi_1X])\xrightarrow{\xi_{m+1}(X)}\mc{S}^{per,s}_{m+1}(X)\\
&\xrightarrow{\eta_{m+1}(X)}H_m(X;L(\Z))\xrightarrow{A_m(X)}L^s_m(\Z[\pi_1X])\rightarrow\cdots
\end{align*}
The groups $\mc{S}^{per,s}_m(X)$ are called the periodic simple structure groups.
One can, for instance, define the assembly map $A_m(X)$ and define $\mc{S}^{per,s}_m(X)$ to be the homotopy groups of the fiber.
When $X$ is an $m$-dimensional simple Poincar\'e complex, elements of $\mc{S}^{per,s}_{m+1}(X)$ can be represented by homotopy equivalences $f:N\rightarrow X$ where $N$ is a homology manifold (see \cite{HomologyManifolds}).

If we define
\[
\mc{S}^{per,\langle j\rangle}_m(X):=H_m\left(X\rightarrow\op{pt};\mathbf{L}^{\langle j\rangle}_{\Z}\right)
\]
then we have algebraic surgery exact sequences for all decorations.
The techniques used to compute $\mc{S}^{per,s}_{n+\ell+1}(M)$ will also compute $\mc{S}^{per,\langle j\rangle}_{n+\ell+1}(M)$ so we state our results in this level of generality.

Let $L(\Z)\langle1\rangle$ denote the $1$-connective cover of $L(\Z)$.
Then, there is an algebraic surgery exact sequence
\begin{align*}
\cdots \rightarrow H_{m+1}(X;L(\Z)\langle1\rangle))\xrightarrow{A^{\langle1\rangle}_{m+1}(X)}&L_{m+1}^s(\Z[\pi_1(X)])\xrightarrow{\xi^{\langle1\rangle}_{m+1}(X)}\mc{S}^{\langle1\rangle,s}_{m+1}(X)\\
&\xrightarrow{\eta^{\langle1\rangle}_{m+1}(X)}H_m(X;L\langle1\rangle)\xrightarrow{A^{\langle1\rangle}_m(X)}L^s_m(\Z[\pi_1(X)])\rightarrow\cdots.
\end{align*}
When $X$ is an $m$ dimensional closed manifold, Ranicki shows that
\[
\mc{S}^{geo,s}(X)\cong\mc{S}^{\langle 1\rangle,s}_{m+1}(X).
\]

Just as we did for the periodic structure sets, we may define structure sets $\mc{S}^{geo,\langle j\rangle}_m(X)$ so that there is an algebraic surgery exact sequence
\[
\cdots\rightarrow H_m(X;L(\Z)\langle1\rangle)\xrightarrow{A^{\langle1\rangle}_{m+1}(X)}L_m^{\langle j\rangle}(\Z[\pi_1(X)])\xrightarrow{\xi^{\langle1\rangle}_{m+1}(X)}\mc{S}^{geo,\langle j\rangle}_m(X)\rightarrow\cdots
\]
for all decorations.
We opt to use the notation $\mc{S}^{geo,\langle j\rangle}_m(X)$ rather than the more appropriate $\mc{S}^{\langle1\rangle,\langle j\rangle}_m(X)$ in order to avoid confusion.

\subsection{The Periodic Structure Set of $B\Gamma  $}
Here, we will record computations of periodic structure sets found in \cite{DavisLuckTorusBundles} and \cite{LuckRosenthal}.

The following results are \cite[Theorem 6.1]{DavisLuckTorusBundles} and \cite[Theorem 1.13]{LuckRosenthal}, respectively.

\begin{theorem}\label{thm: structures on BP}
Let $p$ be an odd prime and let $P$ be a finite $p$-group.
Then, the homomorphism $\xi_*(BP):L^{\langle j\rangle}_*(\Z[P])\rightarrow\mc{S}_*^{per,\langle j\rangle}(BP)$ induces a $\frac{1}{p}$-localization
\[
\tilde{\xi}_*(BP):\tilde{L}^{\langle j\rangle}_*(\Z[P])\rightarrow\mc{S}_*^{per,\langle j\rangle}(BP).
\]
In particular,
\[
\mc{S}_m^{per,\langle j\rangle}(BP)\cong\tilde{L}_m^{\langle j\rangle}(\Z[P])\left[\frac{1}{p}\right]\cong\begin{cases}
\Z\left[\frac{1}{p}\right]^{(p-1)/2}&m\text{ even}\\
0&m\text{ odd}
\end{cases}
\]
when $j\neq 1$.
\end{theorem}

\begin{theorem}\label{thm: structures on BGamma  }
There is an isomorphism
\[
\bigoplus_{(P)\in\mc{P}}\bigoplus_{i=0}^{b+c}\mc{S}^{per,\langle j-i\rangle}_{m-i}(BP)^{\binom{b+c}{i}}\cong\mc{S}^{per,\langle j\rangle}_m(B\Gamma ).
\]
\end{theorem}

\begin{remark}
Although \cite[Theorem 6.1]{DavisLuckTorusBundles} is only stated for the decoration $s$, the proof is valid for all decorations.
\end{remark}

\begin{remark}
In the proof of \cite[Theorem 1.13]{LuckRosenthal}, the authors use an isomorphism $\bigoplus_{P\in \mc{P}}\bigoplus_{i=0}^{b+c}\tilde{K}_j(\Z[P])^{\binom{b+c}{i}}\cong \tilde{K}_j(\Z[\Gamma])$ rather than the correct isomorphism
\[
\bigoplus_{(P)\in\mc{P}}\bigoplus_{i=0}^{b+c}\tilde{K}_{j-i}(\Z[P])^{\binom{b+c}{i}}\cong\tilde{K}_j(\Z[\Gamma])
\]
and conclude that there is an isomorphism
\[
\bigoplus_{(P)\in\mc{P}}\bigoplus_{i=0}^{b+c}\mc{S}^{per,\langle j\rangle}_{n-i}(BP)^{\binom{b+c}{i}}\cong\mc{S}_n^{per,\langle j\rangle}(B\Gamma)
\]
rather than the isomorphism in Theorem \ref{thm: structures on BGamma  }.
The isomorphism in Theorem \ref{thm: structures on BGamma  } can be obtained from the proof of \cite[Theorem 1.13]{LuckRosenthal} after resolving the indexing issue.
\end{remark}

\begin{remark}
The isomorphism in Theorem \ref{thm: structures on BGamma  } can be rewritten as
\[
\mc{S}^{per,\langle j\rangle}_m(B\Gamma)\cong\bigoplus_{(P)\in\mc{P}}L^{\langle j\rangle}_m\left(\Z\left[N_\Gamma P\right]\right)/L^{\langle j\rangle}_m\left(\Z\left[W_\Gamma P\right]\right)\left[\frac{1}{p}\right].
\]
\end{remark}

\subsection{The Periodic Structure Set of $M$}

In order to compute the periodic structure sets $\mc{S}^{per,\langle j\rangle}_{n+\ell+1}(M)$, we follow \cite[\S 8]{DavisLuckTorusBundles}.
Namely, we study a map
\[
\sigma:\mc{S}^{per,\langle j\rangle}_{n+\ell+1}(M)\rightarrow H_n\left(T_{\rho}^n;L(\Z)\right)^{\Z/p}
\]
so that we get a well behaved injection
\[
\sigma\times\mc{S}^{per,\langle j\rangle}_{n+\ell+1}(f):\mc{S}^{per,\langle j\rangle}_{n+\ell+1}(M)\rightarrow H_n\left(T_{\rho}^n;L(\Z)\right)^{\Z/p}\times\mc{S}^{per,\langle j\rangle}_{n+\ell+1}(B\Gamma )
\]
where $f:M\rightarrow B\Gamma $ is the inclusion.

First, we record a consequence of Corollary \ref{cor: K-homology differentials vanish}.
\begin{prop}\label{prop: L-homology differentials vanish}
The differentials in the Atiyah-Hirzebruch-Serre spectral sequences
\begin{align*}
E^2_{i,j}(B\Gamma)=H_i\left(\Z/p;H_j\left(T_{\rho}^n;L(\Z)\right)\right)&\Rightarrow H_{i+j}(B\Gamma;L(\Z))\\
E^2_{i,j}(M)=H_i^{\Z/p}\left(S^{\ell};H_j\left(T_{\rho}^n;L(\Z)\right)\right)&\Rightarrow H_{i+j}(M;L(\Z))
\end{align*}
vanish.
\end{prop}
\begin{proof}
The proof is similar to the proof of \cite[Lemma 8.3]{DavisLuckTorusBundles} so we give an outline.

We first show that the differentials in the first spectral sequence vanish.
It suffices to show that the differentials vanish after inverting $p$ and after localizing $p$.
After inverting $p$, the only nonzero terms are in the column $E^2_{0,j}$ so the differentials must vanish.

After localizing at $p$ and applying \cite[Theorem 4.2]{DavisLuckTorusBundles}, it suffices to show that the differentials for the homology theory $KO_*(-)_{(p)}$ vanish.
But multiplication by $2$ in the homology theory $KO_*$ factors through the map $K_*$.
This exhibits a $KO_*(B\Gamma)_{(p)}$ as a retract of $K_*(B\Gamma)_{(p)}$.
The result follows from Corollary \ref{cor: K-homology differentials vanish}, which asserts that the differentials in $K_*(B\Gamma)_{(p)}$ vanish.

To show that the differentials in the second spectral sequence vanishes, we consider the map $f:M\rightarrow B\Gamma$.
This induces a map from the second spectral sequence to the first which is bijective on terms $E^2_{i,j}$ for $i<\ell$ and surjective on terms $E^2_{\ell,j}$.
Since the terms $E^2_{i,j}(M)=0$ for $i>\ell$, this implies that the differentials vanish as desired.
\end{proof}

Let $F_{\ell,n}(-)$ and $E^r_{\ell,n}(-)$ denote the filtration terms and $E^r$ terms of the spectral sequences in Proposition \ref{prop: L-homology differentials vanish}.
Note that $F_{\ell,n}(M)=H_{n+\ell}(M;L(\Z))$ because the base space is $\ell$-dimensional.
Note also that there is always a quotient $pr:F_{\ell,n}\rightarrow E^{\infty}_{\ell,n}$.
Finally, Proposition \ref{prop: L-homology differentials vanish} implies that $E^{\infty}_{\ell,n}\cong E^{2}_{\ell,n}$ for both spectral sequences.
This explains the second, third, fourth and fifth rows of the following diagram.

\[
\begin{tikzpicture}[scale=2]
\node (A) at (0,7) {$\mc{S}^{per,\langle j\rangle}_{n+\ell+1}(M)$};\node (B) at (3,7) {$\mc{S}^{per,\langle j\rangle}_{n+\ell+1}(B\Gamma )$};
\node (C) at (0,6) {$H_{n+\ell}(M;L(\Z))$};\node (D) at (3,6) {$H_{n+\ell}(B\Gamma ;L(\Z))$};
\node (E) at (0,5) {$F_{\ell,n}(M)$};\node (F) at (3,5) {$F_{\ell,n}(B\Gamma )$};
\node (G) at (0,4) {$E_{\ell,n}^{\infty}(M)$};\node (H) at (3,4) {$E_{\ell,n}^{\infty}(B\Gamma )$};
\node (I) at (0,3) {$E_{\ell,n}^2(M)$};\node (J) at (3,3) {$E_{\ell,n}^2(B\Gamma )$};
\node (K) at (0,2) {$H_{\ell}^{\Z/p}\left(S^{\ell};H_n\left(T_{\rho}^n;L(\Z)\right)\right)$};\node (L) at (3,2) {$H_{\ell}^{\Z/p}(E\Z/p;H_n(T_{\rho};L(\Z))$};
\node (M) at (0,1) {$H_n\left(T_{\rho}^n;L(\Z)\right)^{\Z/p}$};\node (N) at (3,1) {$H_{\ell}\left(\Z/p;H_n\left(T_{\rho}^n;L(\Z)\right)\right)$};
\node (O) at (0,0) {$L_n\left(\Z\left[\Z^n_{\rho}\right]\right)^{\Z/p}$};\node (P) at (3,0) {$H_{\ell}\left(\Z/p;H_n\left(T_{\rho}^n;L(\Z)\right)\right)$};
\path[->] (A) edge node[above]{$\mc{S}^{per,s}_{n+\ell+1}(f)$} (B) (A) edge node[right]{$\eta_{n+\ell+1}(M)$} (C) (B) edge node[right]{$\eta_{n+\ell+1}(B\Gamma )$} (D)
(C) edge node[above]{$H_{n+\ell}(f;L(\Z))$} (D)
(E) edge node[right]{$\op{inc}$} node[left]{$\cong$} (C) (F) edge node[right]{$\op{inc}$}(D) (E) edge node[above]{$F_{\ell,n}(f)$} (F) (E) edge node[right]{$\op{pr}$} (G) (F) edge node[right]{$\op{pr}$} (H)
(G) edge node[above]{$E_{\ell,n}^{\infty}(f)$} (H) (G) edge node[right]{$\op{id}$} node[left]{$\cong$} (I) (H) edge node[right]{$\op{id}$} node[left]{$\cong$} (J)
(I) edge node[above]{$E_{\ell,n}^2(f)$} (J) (I) edge node[right]{$\op{id}$} node[left]{$\cong$} (K) (J) edge node[right]{$\op{id}$} node[left]{$\cong$} (L)
(K) edge node[above]{$g_{n+\ell}$} (L) (K) edge node[right]{$\op{id}$} node[left]{$\cong$} (M) (L) edge node[right]{$\op{id}$} node[left]{$\cong$} (N)
(M) edge node[above]{$g_{n+\ell}$} (N) (M) edge node[right]{$A_n\left(T_{\rho}^n\right)^{\Z/p}$} node[left]{$\cong$} (O) (N) edge node[right]{$\op{id}$} node[left]{$\cong$} (P)
(O) edge node[above]{$g_{n+\ell}\circ A_n(T_{\rho})^{\Z/p}$} (P);
\end{tikzpicture}
\]

The maps $\eta_{n+\ell+1}$ are from the surgery exact sequence and $A_n\left(T_{\rho}^n\right)^{\Z/p}$ is the assembly map.
We define $\mu$ to be the composite of the left vertical maps and we define $\sigma$ to be the composite of $\mu$ with the isomorphism $L_n\left(\Z\left[\Z^n_{\rho}\right]\right)^{\Z/p}\cong H_n(T_{\rho}^n;L(\Z))^{\Z/p}$.
Our goal now is to show the following.

\begin{lemma}\label{lem: first map iso on kernel}
The map $\eta_{n+\ell+1}:\mc{S}^{per,\langle j\rangle}_{n+\ell+1}(M)\rightarrow H_{n+\ell}(M;L(\Z))$ induces an isomorphism
\[
\ker\left(\mc{S}^{per,\langle j\rangle}_{n+\ell+1}(f)\right)\rightarrow\ker\left(H_{n+\ell}(f;L(\Z))\right).
\]
\end{lemma}
\begin{proof}
Consider the following commutative diagram.
\[
\begin{tikzpicture}[scale = 2]
\node (A) at (0,4) {$H_{n+\ell+1}(M;L(\Z))$};\node (B) at (3,4) {$H_{n+\ell+1}(B\Gamma ;L(\Z))$};
\node (C) at (0,3) {$L^{\langle j\rangle}_{n+\ell+1}(\Z[\Gamma])$};\node (D) at (3,3) {$L^{\langle j\rangle}_{n+\ell+1}(\Z[\Gamma])$};
\node (E) at (0,2) {$\mc{S}^{per,{\langle j\rangle}}_{n+\ell+1}(M)$};\node (F) at (3,2) {$\mc{S}^{per,\langle j\rangle}_{n+\ell+1}(B\Gamma )$};
\node (G) at (0,1) {$H_{n+\ell}(M;L(\Z))$};\node (H) at (3,1) {$H_{n+\ell}(B\Gamma ;L(\Z))$};
\node (I) at (0,0) {$L_{n+\ell}^{\langle j\rangle}(\Z[\Gamma])$};\node (J) at (3,0) {$L^{\langle j\rangle}_{n+\ell}(\Z[\Gamma])$};
\path[->] (A) edge (B) (A) edge (C) (B) edge (D) (C) edge node[above]{$\op{id}$} (D) (C) edge (E) (D) edge (F) (E) edge node[above]{$\mc{S}_{n+\ell+1}^{per,\langle j\rangle}(f)$} (F) (E) edge node[right]{$\eta_{n+\ell+1}(M)$} (G) (F) edge node[right]{$\eta_{n+\ell+1}(B\Gamma)$} (H) (G) edge node[above]{$H_{n+\ell}(f;L(\Z))$} (H) (G) edge (I) (H) edge (J) (I) edge node[above]{$\op{id}$} (J);
\end{tikzpicture}
\]
Surjectivity of the map follows from the bottom three rows.

Now, we show injectivity.
Suppose $\alpha\in\ker\left(\mc{S}_{n+\ell+1}^{per,\langle j\rangle}(f)\right)\cap\ker\left(\eta_{n+\ell+1}(M)\right)$.
It suffices to show that $\alpha=0$ after localizing at $p$ and after inverting $p$.

After localizing at $p$, the top map becomes
\[
KO_{n+\ell+1}(M)_{(p)}\rightarrow KO_{n+\ell+1}(B\Gamma )_{(p)}.
\]
Some diagram chasing shows that $\alpha$ pulls back to an element of $L^{\langle j\rangle}_{n+\ell+1}(\Z[\Gamma])$ which then pulls back to an element of $\beta\in KO_{n+\ell+1}(B\Gamma )_{(p)}$.
It suffices to show that $\beta$ is in the image of an element of $KO_{n+\ell+1}(M)_{(p)}$.
Since $L_{n+\ell+1}^{\langle j\rangle}(\Z[\Gamma])$ is $p$-torsion free, it follows that $\beta$ does not have $p$-power order.
Now, consider the diagram
\[
\begin{tikzpicture}[scale =2]
\node (A) at (0,1) {$KO_{n+\ell+1}\left(T_\rho^n\times S^{\ell}\right)_{(p)}$};\node (B) at (2,1) {$KO_{n+\ell+1}(M)_{(p)}$};
\node (C) at (0,0) {$KO_{n+\ell+1}\left(T_\rho^n\right)_{(p)}$};\node (D) at (2,0) {$KO_{n+\ell+1}(B\Gamma )_{(p)}$};
\path[->] (A) edge (B) (A) edge (C) (B) edge (D) (C) edge (D);
\end{tikzpicture}.
\]
The left vertical map is a surjection as it has a section.
The bottom map surjects onto the $p$-torsion free part of $KO_{n+\ell+1}(B\Gamma )_{(p)}$ by Corollary \ref{cor: KO-homology of BGamma }.
Thus, the right vertical map surjects onto the $p$-torsion free part of $KO_{n+\ell+1}(B\Gamma )_{(p)}$.
This gives the desired result.

After inverting $p$, we pull $\alpha$ back to an element $\beta\in H_{n+\ell+1}(B\Gamma;L(\Z))\left[\frac{1}{p}\right]$.
It suffices to prove that
\[
H_{n+\ell+1}\left(M;L(\Z)\right)\left[\frac{1}{p}\right]\rightarrow H_{n+\ell+1}\left(B\Gamma;L(\Z)\right)\left[\frac{1}{p}\right]
\]
is surjective.
Consider the following diagram.
\[
\begin{tikzpicture}[scale=2]
\node (A) at (0,1) {$H_{n+\ell+1}\left(T_\rho^n\times S^{\ell};L(\Z)\right)\left[\frac{1}{p}\right]$};\node (B) at (3,1) {$H_{n+\ell+1}\left(M;L(\Z)\right)\left[\frac{1}{p}\right]$};
\node (C) at (0,0) {$H_{n+\ell+1}\left(T_\rho^n\times S^{\infty};L(\Z)\right)\left[\frac{1}{p}\right]$};\node (D) at (3,0) {$H_{n+\ell+1}\left(B\Gamma ;L(\Z)\right)\left[\frac{1}{p}\right]$};
\path[->] (A) edge (B) (A) edge (C) (B) edge (D) (C) edge (D);
\end{tikzpicture}
\]
Again, the left vertical map is surjective since it has a section.
It remains to show that the bottom map is surjective.
For this, note that the map factors through the isomorphism
\[
\left(H_{n+\ell+1}\left(T_\rho^n\times S^{\infty};L(\Z)\right)\left[\frac{1}{p}\right]\right)_{\Z/p}\rightarrow H_{n+\ell+1}(B\Gamma ;L(\Z))\left[\frac{1}{p}\right].
\]
Therefore, the right vertical map is surjective, which completes the proof.
\end{proof}

\begin{lemma}\label{lem: second map iso on kernel}
The composite 
\[
H_{n+\ell}(M;L(\Z))\rightarrow F_{\ell,n}(M)\rightarrow E_{\ell,n}^{\infty}(M)
\]
induces an isomorphism
\[
\ker\left(H_{n+\ell}(f;L(\Z))\right)\rightarrow\ker\left(E^{\infty}_{\ell,n}(f)\right).
\]
\end{lemma}

\begin{proof}
The proof is the same as the proof of \cite[Lemma 8.5]{DavisLuckTorusBundles}.
\end{proof}

From Lemma \ref{lem: first map iso on kernel} and Lemma \ref{lem: second map iso on kernel}, we conclude that $\sigma\times\mc{S}^{per,\langle j\rangle}_{n+\ell+1}(f)$ is injective.

\begin{prop}\label{prop: product map is injective}
The map
\[
\sigma\times\mc{S}^{per,\langle j\rangle}_{n+\ell+1}(f):\mc{S}^{per,\langle j\rangle}_{n+\ell+1}(M)\rightarrow H_n\left(T^n_{\rho};L(\Z)\right)^{\Z/p}\times\mc{S}^{per,\langle j\rangle}_{n+\ell+1}(B\Gamma)
\]
is injective.
\end{prop}

\begin{prop}\label{prop: structure set cokernel p-torsion}
The cokernel of $\sigma$ is a finite $p$-group.
\end{prop}
\begin{proof}
As in the proof of \cite[Theorem 8.1(2)]{DavisLuckTorusBundles} we consider the following diagram.
\[
\begin{tikzpicture}[scale=2]
\node (A) at (0,2) {$F_{\ell-1,n+1}(M)$};\node (B) at (2,2) {$F_{\ell-1,n+1}(B\Gamma )$};
\node (C) at (0,1) {$H_{n+\ell}(M;L(\Z))$};\node (D) at (2,1) {$H_{n+\ell}(B\Gamma ;L(\Z))$};
\node (E) at (0,0) {$L_{n+\ell}^{\langle j\rangle}(\Z[\Gamma])$};\node (F) at (2,0) {$H_{n+\ell}(\underline{B}\Gamma ;L(\Z))$};
\path[->] (A) edge node[above]{$\cong$} (B) (A) edge (C) (B) edge node[right]{$\cong_{\frac{1}{p}}$} (D) (C) edge (D) (C) edge node[left]{$A_{n+\ell}(M)$} (E) (D) edge node[right]{$\cong_{\frac{1}{p}}$} (F) (E) edge node[above]{$\beta$} (F) (D) edge node[above left]{$A_{n+\ell}(B\Gamma)$} (E);
\end{tikzpicture}
\]
Unlike \cite{DavisLuckTorusBundles}, the bottom map is only an surjection which splits after inverting $p$.

After inverting $p$, we have the following diagram (where inverting $p$ is omitted from the notation).
\begin{equation}\label{diag: structure set cokernel p-torsion}
\begin{tikzpicture}[scale=1.5]
\node (A) at (3,4) {$0$};\node (B) at (6,4) {$0$};
\node (C) at (3,3) {$F_{\ell-1,n+1}(M)$};\node (D) at (6,3) {$H_{n+\ell}(\underline{B}\Gamma ;L(\Z))$};
\node (E) at (0,2) {$\mc{S}^{per,\langle j\rangle}_{n+\ell+1}(M)$};\node (F) at (3,2) {$H_{n+\ell}(M;L(\Z))$};\node (G) at (6,2) {$L_{n+\ell}^{\langle j\rangle}(\Z[\Gamma])$};
\node (H) at (3,1) {$L_n\left(\Z\left[\Z^n_{\rho}\right]\right)^{\Z/p}$};
\node (I) at (3,0) {$0$};
\path[->] (A) edge (C) (B) edge (D) (C) edge node[above]{$\cong_{\frac{1}{p}}$} (D) (C) edge (F) (E) edge node[above]{$\eta_{n+\ell+1}$} (F) (E) edge node[below left]{$\mu$} (H) (F) edge node[above]{$A_{n+\ell}(M)$} (G) (D) edge node[right]{$\gamma$} (G) (F) edge node[right]{$A_n\left(T_{\rho}\right)^{\Z/p}\circ pr$} (H) (H) edge (I);
\end{tikzpicture}
\end{equation}
The image of $A_{n+\ell}(M)$ is contained in the image of $\gamma$ since the assembly map factors through $H_{n+\ell}(M;L(\Z))\rightarrow H_{n+\ell}(B\Gamma ;L(\Z))$.
Now, diagram chasing gives that $\mu\left[\frac{1}{p}\right]$ is surjective.
This implies that $\sigma\left[\frac{1}{p}\right]$ is surjective, which completes the proof.
\end{proof}

\begin{prop}\label{prop: structures come from normalizers}
Let $v$ be the composite
\[
v:\bigoplus_{(P)\in\mc{P}}L^{\langle j\rangle}_{n+\ell+1}\left(\Z\left[N_\Gamma P\right]\right)/L^{\langle j\rangle}_{n+\ell+1}\left(\Z\left[W_\Gamma P\right]\right)\rightarrow \tilde{L}_{n+\ell+1}^{\langle j\rangle}(\Z[\Gamma])\xrightarrow{\xi_{n+\ell+1}(M)} \mc{S}^{per,\langle j\rangle}_{n+\ell+1}(M)
\]
where the first map comes from Equation (\ref{eq: L split with deco}) and the second map comes from the surgery exact sequence.
Then, $v$ is injective, $\op{im}(v)\subseteq\ker(\sigma)$ and $\ker(\sigma)/\op{im}(v)$ is a finite abelian $p$-group.
\end{prop}
\begin{proof}
Consider the commutative diagram
\[
\begin{tikzpicture}[scale =2]
\node (A) at (3,4) {$0$};\node (B) at (3,3) {$\bigoplus_{(P)\in\mc{P}}L^{\langle j\rangle}_{n+\ell+1}\left(\Z\left[N_\Gamma P\right]\right)/L^{\langle j\rangle}_{n+\ell+1}\left(\Z\left[W_\Gamma P\right]\right)$};
\node (C) at (0,2) {$H_{n+\ell+1}(B\Gamma;L(\Z))$};\node (D) at (3,2) {$L^{\langle j\rangle}_{n+\ell+1}(\Z[\Gamma])$};\node (E) at (6,2) {$\mc{S}^{per,\langle j\rangle}_{n+\ell+1}(M)$};
\node (F) at (3,1) {$H_{n+\ell+1}(\underline{B}\Gamma;L(\Z))$};
\node (G) at (3,0) {$0$};
\path[->] (A) edge (B) (B) edge (D) (B) edge node[above right]{$v$} (E) (C) edge node[above]{$A_{n+\ell+1}(B\Gamma)$} (D) (D) edge node[above]{$\xi_{n+\ell+1}(M)$} (E) (C) edge node[below left]{$H_{n+\ell+1}(i;L(\Z))$} (F) (D) edge (F) (F) edge (G);
\end{tikzpicture}
\]
where the column is split exact after inverting $p$.
In the proof of Lemma \ref{lem: first map iso on kernel} we showed that $H_{n+\ell+1}(M;L(\Z))\rightarrow H_{n+\ell+1}(B\Gamma;L(\Z))$ is surjective after inverting $p$.
Since the assembly map $A_{n+\ell+1}(M)$ factors through $A_{n+\ell+1}(B\Gamma)$, the row is exact after inverting $p$.
Moreover, after inverting $p$, the map $H_{n+\ell+1}(B\Gamma;L(\Z))\rightarrow H_{n+\ell+1}(\underline{B}\Gamma;L(\Z))$ is an isomorphism.

We first show that $v$ is injective.
Since $\bigoplus_{(P)\in\mc{P}}L^{\langle j\rangle}_{n+\ell+1}\left(\Z\left[N_\Gamma P\right]\right)/L^{\langle j\rangle}_{n+\ell+1}\left(\Z\left[W_\Gamma P\right]\right)$ is $p$-torsion free, it suffices to show that $v$ is injective after inverting $p$.
But this follows from the splitting in the vertical column of the diagram above and exactness of the row.

We now show that $\op{im}(v)\subseteq\ker(\sigma)$ and that $\ker(\sigma)/\op{im}(v)$ is a finite abelian $p$-group.
From the surgery exact sequence, we get that $\op{im}(v)\subseteq \ker(\eta_{n+\ell+1}(M))\subseteq \ker(\mu)$ so it suffices to show that $\ker(\eta_{n+\ell+1}(M))/\op{im}(v)$ and $\ker(\mu)/\ker(\eta_{n+\ell+1}(M))$ are finite abelian $p$-groups.
Some diagram chasing shows that the cokernel of $H_{n+\ell+1}(i;L(\Z))$ is isomorphic to $\op{im}(\xi_{n+\ell+1}(M)/\op{im}(v)$.
Since $H_{n+\ell+1}(\underline{B}\Gamma;L(\Z))$ is a finitely generated abelian group, we see that $\op{im}(\xi_{n+\ell+1}(B\Gamma)/\op{im}(v)$ is a finitely generated abelian $p$-group.

It follows from Diagram (\ref{diag: structure set cokernel p-torsion}) that $\ker\left(\eta_{n+\ell+1}\left[\frac{1}{p}\right]\right)=\ker\left(\mu\left[\frac{1}{p}\right]\right)$.
From this, we see that $\ker(\mu)/\op{im}(\xi_{n+\ell+1}(M))$ is a finite abelian $p$-group.
We have shown that $\ker(\mu)/\op{im}(v)$ is a finite abelian $p$-group.
\end{proof}

\begin{prop}\label{prop: per structure set inverting p}
After inverting $p$, the map
\[
\sigma\times\mc{S}_{n+\ell+1}^{per,\langle j\rangle}(f):\mc{S}_{n+\ell+1}^{per,\langle j\rangle}(M)\rightarrow H_n\left(T_{\rho}^n;L(\Z)\right)^{\Z/p}\times\mc{S}_{n+\ell+1}^{per,\langle j\rangle}(B\Gamma )
\]
is an isomorphism.
\end{prop}
\begin{proof}
It follows from Proposition \ref{prop: product map is injective} that the map is injective.
Surjectivity follows from Proposition \ref{prop: structure set cokernel p-torsion} and Proposition \ref{prop: structures come from normalizers}.
\end{proof}

Now, using Proposition \ref{prop: product map is injective}, Proposition \ref{prop: per structure set inverting p} and the fact that $H_n\left(T_{\rho}^n;L(\Z)\right)^{\Z/p}\times\mc{S}^{per,\langle j\rangle}_{n+\ell+1}(B\Gamma)$ has no $p$-torsion, we obtain an integral computation.

\Sper

\subsection{The Geometric Simple Structure Set of $M$}
Identifying $\mc{S}^{geo,s}(M)$ with $\mc{S}^{\langle1\rangle,s}_{n+\ell+1}(M)$, we see that there is a map
\[
j(M):\mc{S}^{geo,s}(M)\rightarrow\mc{S}^{per,s}(M).
\]

The proof of \cite[Theorem 9.2]{DavisLuckTorusBundles} and the results above give the following theorem.

\begin{theorem}\label{thm: computation of geometric simple structure set}
There is a homomorphism
\[
\sigma^{geo}:\mc{S}^{geo,s}(M)\rightarrow H_n(T_{\rho};L(\Z)\langle1\rangle)^{\Z/p}
\]
such that the following hold:
\begin{enumerate}
\item The map
\[
\sigma^{geo}\times\left(\mc{S}^{per,s}_{n+\ell+1}(f)\circ j(M)\right):\mc{S}^{geo,s}(M)\rightarrow H_n(T_{\rho};L(\Z)\langle1\rangle)^{\Z/p}\times\mc{S}_{n+\ell+1}^{per,s}(B\Gamma )
\]
is injective.
\item The cokernel of $\sigma^{geo}$ is a finite abelian $p$-group.
\item Consider the composite
\[
\nu^{geo}:\bigoplus_{(P)\in\mc{P}}H_{n+\ell+1}\left(T^b;\tilde{L}^s(\Z[P])\right)\rightarrow \tilde{L}^s_{n+\ell+1}(\Z[\Gamma])\xrightarrow{\tilde{\xi}^{\langle1\rangle}_{n+\ell+1}}\mc{S}^{geo,s}(M)
\]
where the first map comes from (\ref{eq: L split with deco}) and $\tilde{\xi}^{\langle1\rangle}_{n+\ell+1}$ is the map from the geometric surgery exact sequence.
Then $\nu^{geo}$ is injective, the image of $\nu^{geo}$ is contained in the kernel of $\sigma^{geo}$ and $\ker(\sigma^{geo})/\op{im}(\nu^{geo})$ is a finite abelian $p$-group.
\item After inverting $p$, the map $\sigma^{geo}\times\left(\mc{S}^{per,s}_{n+\ell+1}(f)\circ j(M)\right)$ is an isomorphism.
\end{enumerate}
\end{theorem}

From this, we conclude
\Sgeo

\bibliographystyle{alpha}
\bibliography{TorusBundlesLensSpacesFinal}
\end{document}